\documentclass[a4paper,10pt]{amsart}
\usepackage{amsmath}
\usepackage{amssymb}
\usepackage{amsthm}
\usepackage{comment}
\usepackage{shuffle}
\usepackage[dvipdfmx]{hyperref}
\usepackage{here}
\usepackage{diagbox}
\usepackage{mathtools}

\numberwithin{equation}{section}

\newtheorem{theorem}{Theorem}
\newtheorem*{theorem*}{Theorem}
\newtheorem{lemma}{Lemma}
\newtheorem*{lemma*}{Lemma}
\newtheorem{corollary}{Corollary}
\newtheorem*{corollary*}{Corollary}

\newtheorem*{propositon*}{Proposition}

\theoremstyle{definition}
\newtheorem{definition}{Definition}

\newtheorem*{definition*}{Definition}
\newtheorem{remark}{Remark}
\newtheorem*{remark*}{Remark}

\newtheorem*{example*}{Example}

\allowdisplaybreaks[1]

\begin{document}
\title[Relations of multiple $\tilde{T}$-values]{Relations of multiple $\tilde{T}$-values involving the total numbers of certain permutations}
\author{RYOTA UMEZAWA}
\address{Chubu University\\
1200 Matsumoto-cho, Kasugai-shi, Aichi 487-8501, JAPAN.}
\email{cu38599@fsc.chubu.ac.jp}
\date{}
\maketitle

\begin{abstract}
Kaneko and Tsumura proved a relation of multiple $\tilde{T}$-values involving Entringer numbers counting the total number of down-up permutations starting with a fixed value. In the present paper, we generalize this relation and provide some relations involving Entringer numbers and the total number of Dumont permutations of the first kind starting with a fixed value. For this purpose, we also provide explicit formulas for the total numbers of those permutations.
\end{abstract}
\section{Introduction}\label{se:intro}
Multiple ${T}$-values and Multiple $\tilde{T}$-values are introduced by Kaneko and Tsumura as follows:
\begin{definition}[Multiple $T$-values (Kaneko-Tsumura \cite{KT1, KT})] For an index $\mathbf{k}=(k_{1},\dots,k_{r}) \in \mathbb{Z}^{r}_{ \ge 1}$ with $k_{r} > 1$,
\begin{align*}
T(\mathbf{k}) = 2^{r}\sum_{\substack{0<m_{1}<\cdots<m_{r}\\m_{j}\equiv j\ {\rm mod}\ 2}}\frac{1}{m^{k_{1}}_{1}\cdots m^{k_{r}}_{r}}.
\end{align*}
\end{definition}
\begin{definition}[Multiple $\tilde{T}$-values (Kaneko-Tsumura \cite{KT3})] For an index $\mathbf{k}=(k_{1},\dots,k_{r}) \in \mathbb{Z}^{r}_{ \ge 1}$,
\begin{align*} 
\tilde{T}(\mathbf{k}) = 2^{r}\sum_{\substack{0<m_{1}<\cdots<m_{r}\\m_{i}\equiv i\ {\rm mod}\ 2}}\frac{(-1)^{(m_{r}-r)/2}}{m^{k_{1}}_{1}\cdots m^{k_{r}}_{r}}.
\end{align*}
\end{definition}
We focus on the remarkable result
\begin{align}\label{eq:KTrel}
\sum_{j=1}^{n} \mathbb{E}(n,j)\, \tilde{T}(\underbrace{1,\dots,1,\overset{\tiny j}{\check{2}},1,\dots,1}_{n}) =  
\begin{cases} 
\tilde{T}(n+1) &(n: \textup{odd}),\\
T(n+1)& (n: \textup{even})
\end{cases}
\end{align}
proved in \cite{KT3}.
Here, $\mathbb{E}(n,j)$ denotes the Entringer numbers, which counts the number of
down-up permutations (alternating permutations) i.e.,
\[\sigma(i) < \sigma(i+1) \iff i \text{ is even},\]
in the symmetric group $\mathfrak{S}_{n+1}$ starting with $j + 1$. For example,
\[\mathbb{E}(4,1) = |\{(\sigma(1),\dots,\sigma(5)) = (2,1,5,3,4), (2,1,4,3,5)\}| = 2.
\]
In the present paper, we generalize the relation \eqref{eq:KTrel} and provide some relations of multiple $\tilde{T}$-values, which involve not only the Entringer numbers but also the total number of Dumont permutations (of the first kind). Dumont permutations, named after Dumont's study \cite{D}, are defined as permutations where the descents occur after the even numbers and the ascents after the odd numbers i.e.,
\[\sigma(i) < \sigma(i+1) \iff \sigma(i) \text{ is odd}.\]
Let $\mathbb{G}(k,j)$ be the total number of Dumont permutations in $\mathfrak{S}_{k-1}$ starting with $j$. For example,
\[\mathbb{G}(5,2) = |\{(\sigma(1),\dots,\sigma(4)) = (2,1,3,4), (2,1,4,3)\}| = 2.
\]
Then, the main theorems obtained in this paper are listed as follows:
\begin{theorem}\label{th:Tt+T=EG}
When $k_{1} \ge 1$ and $k_{2} \ge 3$ are both odd, we have
\begin{align*}
&\sum_{l=1}^{k_{1}} \sum_{m=1}^{k_{2}} \binom{k_{1}-l+m}{m} \binom{k_{2}-m+l}{l} \mathbb{E}(k_{1},l)\, \mathbb{G}(k_{2}, m)\, \tilde{T}(\underbrace{1,\dots,1,\hspace{-4mm}\overset{\tiny k_{1}-l+m}{\check{2}}\hspace{-4mm},1,\dots,1}_{k_{1}+k_{2}})\\
&= {T}(k_{1}+1,k_{2}) + \tilde{T}(k_{1}+1,k_{2}).
\end{align*}
\end{theorem}
\begin{theorem} \label{th:Tt+T=G} For odd $k_{2} \ge 3$, we have
\begin{align*}
&\sum_{m_{2}=1}^{k_{2}} \mathbb{G}(k_{2},m_{2})\sum_{m_{1} = m_{2}}^{k_{2}}  \tilde{T}(\underbrace{1,\dots,1,1\hspace{-0mm}\overset{\tiny m_{2}}{\check{+1}}\hspace{-0mm},1,\dots,1,1\hspace{-0mm}\overset{\tiny m_{1}}{\check{+1}}\hspace{-0mm},1,\dots,1}_{k_{2}})\\
&= \tilde{T}(1) T(1,k_{2}) + \tilde{T}(1) \tilde{T}(1,k_{2}).
\end{align*}
\end{theorem}
\begin{theorem}\label{th:TorTt1=EE}
For odd $k_{1} \ge 1$, we have
\begin{align*}
&\tilde{T}(1) \sum_{\substack{l_{1}+l_{2}+l_{3} = k_{1} \\ l_{1}>0}} \sum_{\substack{m_{1}+m_{2}+m_{3} = k_{3} \\ m_{3} > 0}} \binom{l_{1}+m_{1}}{l_{1}} \binom{l_{2}+m_{2}}{l_{2}} \binom{l_{3}+m_{3}}{l_{3}}\\
&\quad \times \mathbb{E}(k_{1},l_{1})\,\mathbb{E}(k_{3},m_{3})\, \tilde{T}(\underbrace{1,\dots,1,1\hspace{-2mm}\overset{\tiny l_{3}+m_{3}}{\check{+1}}\hspace{-2mm},1,\dots,1,1\hspace{-7mm}\overset{\tiny l_{3}+m_{3}+l_{2}+m_{2}}{\check{+1}}\hspace{-7mm},1,\dots,1}_{k_{1}+k_{3}})\\
&=\begin{cases}
\tilde{T}({k}_{1}+1,1) \tilde{T}({k}_{3}+1)-\tilde{T}(k_{1}+1,1,k_{3}+1) & (k_{3}:\textup{odd}),\\
\tilde{T}({k}_{1}+1,1) {T}({k}_{3}+1)+{T}(k_{1}+1,1,k_{3}+1) & (k_{3}:\textup{even}).
\end{cases}
\end{align*}
\end{theorem}
\begin{theorem} \label{th:TorTt=E} We have
\begin{align*} 
&\sum_{m_{3}=1}^{k_{3}} \mathbb{E}(k_{3},m_{3})  \sum_{m_{2} = m_{3}}^{k_{2}}  \sum_{m_{1} = m_{2}}^{k_{3}}\tilde{T}(\underbrace{1,\dots,1,1\hspace{-0mm}\overset{\tiny m_{3}}{\check{+1}}\hspace{-0mm},1,\dots,1,1\hspace{-0mm}\overset{\tiny m_{2}}{\check{+1}}\hspace{-0mm},1,\dots,1,1\hspace{-0mm}\overset{\tiny m_{1}}{\check{+1}}\hspace{-0mm},1,\dots,1}_{k_{3}})\\
&=\begin{cases}
\tilde{T}(1,1) \tilde{T}(k_{3}+1) - \tilde{T}(1,1,k_{3}+1) & (k_{3}: \textup{odd}),\\
\tilde{T}(1,1) {T}(k_{3}+1) + T(1,1,k_{3}+1)& (k_{3}: \textup{even}).
\end{cases}
\end{align*}
\end{theorem}
\begin{theorem} \label{th:TorTt2=EE}For even $k_{1} \ge 0$, we have
\begin{align*}
& \tilde{T}(1) \sum_{\substack{l_{1}+l_{2}+l_{3}= k_{1}}}  \sum_{\substack{m_{1}+m_{2}+m_{3} = k_{3} \\ m_{3} > 0}} \binom{l_{1}+m_{1}}{l_{1}} \binom{1+l_{2}+m_{2}}{1,l_{2},m_{2}} \binom{l_{3}+m_{3}}{l_{3}}  \\
&\quad \times \mathbb{E}(k_{1},k_{1}-l_{1})\, \mathbb{E}(k_{3},m_{3})\,\tilde{T}(\underbrace{1,\dots,1,\hspace{-3mm}\overset{\tiny l_{3}+m_{3}}{\check{2}}\hspace{-3mm},1,\dots,1,\hspace{-10mm}\overset{\tiny l_{3}+m_{3}+l_{2}+m_{2}+1}{\check{2}}\hspace{-10mm},1,\dots,1}_{k_{1}+k_{3}+1})\\
&=\begin{cases}
\tilde{T}(k_{1}+1,2) \tilde{T}(k_{3}+1) - \tilde{T}(k_{1}+1,2,k_{3}+1) & (k_{3}: \textup{odd}),\\
\tilde{T}(k_{1}+1,2) {T}(k_{3}+1) + T(k_{1}+1,2,k_{3}+1) & (k_{3}: \textup{even}).
\end{cases}
\end{align*}
\end{theorem}
The original proof of \eqref{eq:KTrel} by Kaneko and Tsumura relies on induction and seems difficult to generalize. Therefore, we generalize \eqref{eq:KTrel} using generating functions. This method can be regarded as a generalisation of the method used in the author's previous works \cite{U1} and \cite{U3}. This method yields relations for general indices, but simple relations can only be obtained for some indices. Thus, we derive general relations and prove the main theorems by specializing them. In the process, sums of the form $\sum_{n=1}^{j} \binom{j}{n} a_{k-n}$ appear. We find that when $a_{n}$ is replaced by a certain sequence, $\sum_{n=1}^{j} \binom{j}{n} a_{k-n}$ is equal to the total number of down-up permutations or Dumont permutations starting with a fixed value. Therefore, the total numbers of these permutations appear in the main theorems.

In Section \ref{se:DD}, we will see that the total numbers of down-up permutations and Dumont permutations starting with a fixed value are evaluated by the sum of the form $\sum_{n=1}^{j} \binom{j}{n} a_{k-n}$. In Section \ref{se:relTt}, we will generalize \eqref{eq:KTrel} and give the proofs of main theorems. Appendix \ref{se:Appendix} contains the tables of numbers dealt with in Section \ref{se:DD}, including $\mathbb{E}(n,j)$ and $\mathbb{G}(n,j)$.

\section{Down-up permutations and Dumont permutations} \label{se:DD}
In the present section, we evaluate down-up permutations and Dumont permutations starting with a fixed value by the sum of the form $\sum_{n=1}^{j} \binom{j}{n} a_{k-n}$.
\subsection{Preliminaries}
This paper deals with the following sequences:
\begin{align*}
{B}_{n}:& \textup{ the Bernoulli numbers with $B_{1} = -1/2$},\\
{C}_{n}:&=(2-2^{n})B_{n},\\
{D}_{n}:&= 2^{n}B_{n},\\
{E}_{n}:& \textup{ the Euler numbers},\\
{F}_{n}:&= (2^{n+1}-4^{n+1})B_{n+1}/(n+1),\\
{G}_{n}:&=2(1-2^{n})B_{n} = C_{n}-D_{n}\ \textup{(the Genocchi numbers with $G_{1}=1$)}, \\
\mathbb{E}_{n} :&= \begin{cases} i^{n} E_{n} & (n:\textup{even}),\\
i^{n+1} F_{n}& (n:\textup{odd}),\end{cases}\\
\mathbb{G}_{n} :&= (-1)^{\lfloor n/2 \rfloor} G_{n}= \begin{cases}i^{n} G_{n}& (n:\textup{even}),\\
0& (n\ge3:\textup{odd}),\\
1& (n=1).\end{cases}
\end{align*}
Note that ${B}_{n} = D_{n} = G_{n} = \mathbb{G}_{n} = 0$ for odd $k \ge 3$, ${C}_{n} = E_{n} = 0$ for odd $k \ge 1$, ${F}_{n} = 0$ for even $k \ge 2$. Also note that $\mathbb{E}_{n} \ge 0$ and $\mathbb{G}_{n} \ge 0$ for any $n \ge 0$. These facts are used throughout this paper. It is well-known that the number $\mathbb{E}_{n}$ is equal to the total number of down-up (alternating) permutations in $\mathfrak{S}_n$. Additionally, Dumont \cite{D} has proven that the number $\mathbb{G}_{2n}$ is equal to the total number of Dumont permutations in $\mathfrak{S}_{2n-1}$.

The generating functions of each sequence are as follows:
\begin{align*}
\sum_{n=0}^{\infty} B_{n}\frac{X^{n}}{n!} &= \frac{X}{e^{X}-1},&
\sum_{n=0}^{\infty} {C}_{n}\frac{X^{n}}{n!}  &= \frac{2X}{e^{X}-e^{-X}},\\
\sum_{n=0}^{\infty} {D}_{n}\frac{X^{n}}{n!} &= \frac{2Xe^{-X}}{e^X{-e^{-X}}},&
\sum_{n=0}^{\infty} E_{n}\frac{X^{n}}{n!} &= \frac{2}{e^{X}+e^{-X}},\\
\sum_{n=0}^{\infty} F_{n}\frac{X^{n} }{n!} &= \frac{2e^{-X}}{e^{X}+e^{-X}},&
\sum_{n=0}^{\infty} G_{n}\frac{X^{n}}{n!}  &= \frac{2X}{e^{X}+1},\\
\sum_{n=0}^{\infty} \mathbb{E}_{n}\frac{X^{n}}{n!} &= \sec{X}+\tan{X},&
\sum_{n=0}^{\infty} \mathbb{G}_{n}\frac{X^{n}}{n!} &= {X}+X\tan{\frac{X}{2}}.
\end{align*}
The first ten values of each sequence are as follows:
{
\begin{table}[H]
\centering
  \begin{tabular}{|c|c|c|c|c|c|c|c|c|c|c|c|} \hline
	$n$& 0 & 1 & 2 & 3 & 4 &5 &6 &7&8&9&10\\ \hline
	${B}_{n}$ &1 &-1/2& 1/6& 0& -1/30& 0& 1/42& 0& -1/30& 0& 5/66\\ \hline 
	${C}_{n}$ &1& 0& -1/3& 0& 7/15& 0& -31/21& 0& 127/15& 0& -2555/33\\ \hline 
	${D}_{n}$ &1& -1& 2/3& 0& -8/15& 0& 32/21& 0& -128/15& 0& 2560/33\\ \hline 
	${E}_{n}$ &1& 0 & -1 & 0 & 5 & 0 & -61 & 0 & 1385 & 0 & -50521\\ \hline 
	${F}_{n}$ &1 & -1 & 0& 2& 0& -16& 0& 272& 0& -7936& 0\\ \hline 
	${G}_{n}$ &0& 1& -1& 0& 1& 0& -3& 0& 17& 0& -155\\ \hline 
	$\mathbb{E}_{n}$ &1&1& 1 & 2 & 5 & 16 & 61 & 272 & 1385 &7936&50521\\ \hline 
	$\mathbb{G}_{n}$&0&1 & 1 & 0 & 1 & 0 & 3 & 0 & 17 & 0 &155\\ \hline
  \end{tabular}
\end{table}
}
\subsection{Down-up permutations and Dumont permutations} 
We define $\mathbb{E}(k,j)$ as the Entringer number, that is the total number of down-up permutations in $\mathfrak{S}_{k+1}$ starting with $j+1$, and $\mathbb{G}(k,j)$ as the total number of Dumont permutations in $\mathfrak{S}_{k-1}$ starting with $j$. The numbers $\mathbb{E}(k,j)$ satisfy the following recurrence relation (see \cite{E}, where note that $A(n,k) = \mathbb{E}(n-1,n-k)$):
\begin{align*}
&\mathbb{E}(0,0) = 1, \quad \mathbb{E}(k,0) = 0 \quad (k>0),\\
&\mathbb{E}(k,j) = \mathbb{E} (k, j-1) + \mathbb{E}(k-1,k-j)\quad (1\le j \le k).
\end{align*}
 The numbers $\mathbb{G}(k,j)$ satisfy the following recurrence relation (see \cite{K}):
\begin{align*}
&\mathbb{G}(k,0) = 0 \quad (k \ge 0),\\
&\mathbb{G}(1,1)= \mathbb{G}(2,2) = 0,\quad \mathbb{G}(2,1) = 1,\\
&\mathbb{G}(k,j) = \mathbb{G}(k, j-1) - \sum_{m=0}^{j-2}\mathbb{G}(k-2,m)\quad \textup{($k\ge4$:even, $j\ge1$:odd)},\\
&\mathbb{G}(k,j) = \mathbb{G}(k, j-1) + \sum_{m=j-1}^{k-2}\mathbb{G}(k-2,m)\quad \textup{($k\ge4$:even, $j\ge2$:even)},\\
&\mathbb{G}(k,j) = \mathbb{G}(k+1, j) + \mathbb{G}(k+1,j+1)\quad \textup{($k \ge 3$:odd)}.
\end{align*}
See Appendix \ref{se:Appendix} for the first few rows of $\mathbb{E}(k,j)$ and $\mathbb{G}(k,j)$. 

We particularly focus on the sum $\sum_{n=1}^{j} \binom{j}{n} a_{k-n}$ when $a_{n} = E_{n}$, $F_{n}$ and $G_{n}$, thus we define as follows.
\begin{definition}
For $0 \le j \le k$, we define
\begin{align}
E(k,j) &= (-1)^{\lfloor (k-1)/2 \rfloor}\sum_{n=1}^{j} \binom{j}{n} E_{k-n},\\
F(k,j) &= (-1)^{\lfloor k/2 \rfloor}\sum_{n=1}^{j} \binom{j}{n} F_{k-n},\\
G(k,j) &= (-1)^{\lfloor (k-1)/2 \rfloor}\sum_{n=1}^{j} \binom{j}{n} G_{k-n}.
\end{align}
\end{definition}
These numbers have interesting properties: they are non-negative integers, satisfy certain recurrence relations and have combinatorial interpretation depending on the parity of $k$, which will be proven later in Theorem \ref{th:rrEFGH}, Theorem \ref{th:BE=} and Theorem \ref{th:BG=}.

In addition, to help our discussion we also define the following numbers.
\begin{definition} For $0 \le j \le k$, we define
\begin{align*}
H(k,j) = (-1)^{\lfloor (k-1)/2 \rfloor}\sum_{n=1}^{j} h(j,n)\, G_{k-n},
\end{align*}
where
\begin{align*}
h(j,n) = \sum_{u=n}^{j}(-1)^{j-u}\binom{u-1}{n-1}.
\end{align*}
\end{definition}

See Appendix \ref{se:Appendix} for the first few rows of ${E}(k,j)$, ${F}(k,j)$, ${G}(k,j)$ and ${H}(k,j)$. 

We first provide recurrence relations.
\begin{theorem}\label{th:rrEFGH} For $1\le j \le k$, we have
\begin{align}\label{eq:rrE}
&E(k,j) = E(k,j-1) + (-1)^{k}E(k-1,j-1) +  (-1)^{\lfloor (k-1)/2 \rfloor}E_{k-1}\\ \label{eq:rrF}
&F(k,j) = F(k,j-1) - (-1)^{k}F(k-1,j-1) + (-1)^{\lfloor k/2 \rfloor}{F}_{k-1}\\ \label{eq:rrG}
&G(k,j) = G(k,j-1) + (-1)^{k} G(k-1,j-1) + \mathbb{G}_{k-1},\\ \label{eq:rrH}
&H(k,j) = H(k,j-1) + (-1)^{k} H(k-1,j-1) - (-1)^{j}\mathbb{G}_{k-1}.\\ \label{eq:rrHG}
&H(k,j) = - H(k,j-1) + (-1)^{k} G(k-1,j-1) + \mathbb{G}_{k-1}.
\end{align}
\end{theorem}
\begin{proof}
For $1\le j \le k$, we have
\begin{align*}
(-1)^{k}E(k-1,j-1) &= (-1)^{k}(-1)^{\lfloor (k-2)/2 \rfloor}\sum_{n=1}^{j-1} \binom{j-1}{n} E_{k-1-n}\\
&= (-1)^{\lfloor (k-1)/2 \rfloor}\left(-E_{k-1}+ \sum_{n=1}^{j} \binom{j-1}{n-1} E_{k-n}\right).
\end{align*}
Therefore, we can obtain \eqref{eq:rrE} by using the relation $\binom{j}{n} = \binom{j-1}{n}+\binom{j-1}{n-1}$. In the same way, we can verify \eqref{eq:rrF} and \eqref{eq:rrG}. The proof of \eqref{eq:rrH} uses the fact that $h(j,n)$ satisfies the same recurrence relation as binomial coefficients, namely
\begin{align*}
h(j,n) = h(j-1,n) + h(j-1,n-1) \qquad \text{for $n \ge 2$}.
\end{align*}
This recurrence relation is derived as follows: Noting that $\binom{n-2}{n-1} = 0$, we obtain
\begin{align*}
h(j-1,n) = \sum_{u=n}^{j}(-1)^{j-u}\binom{u-2}{n-1},
\end{align*}
and applying the recurrence relation $\binom{u-1}{n-1} = \binom{u-2}{n-1} + \binom{u-2}{n-2}$ to this equation, we obtain the result.

Applying the recurrence relation of $h(j,n)$ to
\begin{align*}
(-1)^{k}H(k-1,j-1) = (-1)^{\lfloor (k-1)/2 \rfloor}\sum_{n=2}^{j} h(j-1,n-1)\, G_{k-n},
\end{align*}
we have
\begin{align*}
&(-1)^{k}H(k-1,j-1)\\
&= -(-1)^{\lfloor (k-1)/2 \rfloor}(h(j,1) -h(j-1,1))\, G_{k-1} + H(k,j) - H(k,j-1).
\end{align*}
Using
\begin{align*}
(-1)^{\lfloor (k-1)/2 \rfloor}G_{k-1} = \mathbb{G}_{k-1}
\end{align*}
and
\begin{align*}
-(h(j,1) -h(j-1,1) )= -\sum_{u=1}^{j}(-1)^{j-u} + \sum_{u=1}^{j-1}(-1)^{j-1-u} = (-1)^{j},
\end{align*}
we obtain \eqref{eq:rrH}. The proof of \eqref{eq:rrHG} relies on the use of
\begin{align*}
h(j,n) + h(j-1,n) = \binom{j-1}{n-1},
\end{align*}
which is readily derived from the definition. Applying this equation to 
\begin{align*}
(-1)^{k} G(k-1,j-1) 
= - \mathbb{G}_{k-1} + (-1)^{\lfloor (k-1)/2 \rfloor}\sum_{n=1}^{j} \binom{j-1}{n-1} G_{k-n},
\end{align*}
we can obtain \eqref{eq:rrHG}.
\end{proof}
The Entringer number $\mathbb{E}(k,j)$ is equal to $E(k,j)$ or $F(k,j)$ depending on the parity of $k$.
\begin{theorem}\label{th:BE=}
For $k \ge 1$, We have
\begin{align*}
\mathbb{E}(k,j) = \begin{cases}
\displaystyle E(k,j) & (k : \textup{odd}),\\
\displaystyle F(k,j) & (k : \textup{\rm even}).
\end{cases}
\end{align*}
\end{theorem}
\begin{remark}
This is essentially the same as the second equation of Proposition 2 in \cite{KPP}, except for the case $k=j$. They derive this formula from the recurrence relation. Here, we provide a more direct alternative proof.
\end{remark}
\begin{remark}
This theorem also implies that $E(k,j)$ is non-negative when $k$ is odd, and $F(k,j)$ is non-negative when $k$ is even. Furthermore, combining this with Theorem \ref{th:rrEFGH}, it can be concluded that $E(k,j)$ and $F(k,j)$ are non-negative for all $k$.
\end{remark}
\begin{proof}
Since $\mathbb{E}_{n}$ is equal to the total number of down-up permutations in $\mathfrak{S}_n$, when $k \ge 1$, the number of down-up permutations in $\mathfrak{S}_{k+1}$ starting with $j+1$ is equal to 
\begin{align*}
&\mathbb{E}(k, j) \\
&= \left|\left\{\sigma \in \mathfrak{S}_{k+1} \mathrel{}\middle|\mathrel{}  \begin{aligned}
\sigma(1) = j+1 > \sigma(2) < \sigma(3) > \sigma(4) <\cdots \sigma(k+1)
\end{aligned}\right\}\right|\\
&=\sum_{\substack{n=1 \\ n:\text{odd}}}^{j}(-1)^{(n-1)/2}\left|\left\{\sigma \in \mathfrak{S}_{k+1} \mathrel{}\middle|\mathrel{}  \begin{gathered}
\sigma(1)= j+1 > \sigma(2) > \cdots > \sigma(n+1),\\
\sigma(n+2) > \sigma(k+3) < \sigma(k+4) > \cdots \sigma(k+1)
\end{gathered}\right\}\right|\\
&\quad +(-1)^{k/2}\delta_{k: \textup{even}}\left|\left\{\sigma \in \mathfrak{S}_{k+1} \mathrel{}\middle|\mathrel{}  \begin{gathered}
\sigma(1)= j+1 > \sigma(2) > \cdots > \sigma(k+1)
\end{gathered}\right\}\right|\\
&=\begin{dcases}\sum_{\substack{n=1 \\ n:\text{odd}}}^{j}(-1)^{(n-1)/2}\binom{j}{n}\mathbb{E}_{k-n} + (-1)^{ k/2 }& (k=j: \textup{even}),\\
\sum_{\substack{n=1 \\ n:\text{odd}}}^{j}(-1)^{(n-1)/2}\binom{j}{n}\mathbb{E}_{k-n} & \textup{(otherwise)},
\end{dcases}
\end{align*}
where $\delta_{k: \textup{even}} = 1$ (resp. 0) if $k$ is even (resp. odd).
The set in the third line is the set of permutations decreasing from $\sigma(1)$ to $\sigma(n+1)$ and down-up from $\sigma(n+2)$ to $\sigma(k+1)$. 
For example, if $\sigma \in \mathfrak{S}_{k+1}$  is represented by $(\sigma(1), \dots, \sigma(k+1))$, then
\begin{align*}
&\mathbb{E}(4, 4) \\
&= \left|\left\{\begin{aligned}
 (5, 1, 3, 2, 4),
 (5, 1, 4, 2, 3),
 (5, 2, 3, 1, 4),
 (5, 2, 4, 1, 3),
 (5, 3, 4, 1, 2)
\end{aligned}\right\}\right|\\
&= +\left|\left\{\begin{aligned}
& (5, 1, 3, 2, 4),
 (5, 1, 4, 2, 3),
 (5, 2, 3, 1, 4),
 (5, 2, 4, 1, 3),\\
& (5, 3, 2, 1, 4),
 (5, 3, 4, 1, 2),
 (5, 4, 2, 1, 3),
 (5, 4, 3, 1, 2)
\end{aligned}\right\}\right|\\
&\quad -\left|\left\{\begin{aligned}
 (5, 3, 2, 1, 4),
 (5, 4, 2, 1, 3),
 (5, 4, 3, 1, 2),
 (5, 4, 3, 2, 1)
\end{aligned}\right\}\right|\\
&\quad +\left|\left\{\begin{aligned}
(5, 4, 3, 2, 1)
\end{aligned}\right\}\right|\\
&= \binom{4}{1}\mathbb{E}_{3}-\binom{4}{3}\mathbb{E}_{1}+1.
\end{align*}
Finally, using the definition of $E(k, j)$, $F(k, j)$ and $\mathbb{E}_{n}$, we can prove the theorem.
\end{proof}

The following theorem is also used in Section \ref{se:relTt}.
\begin{theorem} \label{th:Ekk-j} We have
\begin{align*}
\mathbb{E}(k,k-j) &= \begin{dcases}
-(-1)^{\lfloor (k-1)/2 \rfloor}\sum_{n=0}^{j} \binom{j}{n} E_{k-n}  & (k : \textup{even}),\\
-(-1)^{\lfloor k/2 \rfloor}\sum_{n=0}^{j} \binom{j}{n} F_{k-n}  & (k : \textup{odd}).
\end{dcases}
\end{align*}
\end{theorem}
\begin{remark}
This is essentially the same as the first equation of Proposition 2 in \cite{KPP}, except for the case $k=j$. A similar equation can be seen in \cite{E}.
\end{remark} 
\begin{proof}
First, note that
\begin{align*}
-(-1)^{\lfloor (k-1)/2 \rfloor}\sum_{n=0}^{j} \binom{j}{n} E_{k-n} &= - E(k,j) -(-1)^{\lfloor (k-1)/2 \rfloor} E_{k}.
\end{align*}
Moreover, putting $k=k+1$ and $j=j+1$ in \eqref{eq:rrF}, we have
\begin{align*}
E(k+1,j+1) = E(k+1,j) + (-1)^{k+1}E(k,j) +  (-1)^{\lfloor k/2 \rfloor}E_{k}.
\end{align*}
Combining the above two equations with $k$ as an even number, we have
\begin{align*}
-(-1)^{\lfloor (k-1)/2 \rfloor}\sum_{n=0}^{j} \binom{j}{n} E_{k-n} &= E(k+1,j+1) - E(k+1,j)\\
&= \mathbb{E}(k+1,j+1) - \mathbb{E}(k+1,j).
\end{align*}
Finally, by the recurrence relation of Entringer numbers, we can obtain the theorem for even $k$. The same method can be used to prove the theorem for odd $k$.
\end{proof}
At the end of this section, we will prove that $\mathbb{G}(k,j)$ is equal to $G(k,j)$ or $H(k,j)$ depending on the parity of $k$.
\begin{theorem}\label{th:BG=}
We have
\begin{align}
\mathbb{G}(k,j)  = \begin{cases}
 G(k,j) & (k : \textup{odd}),\\
 H(k,j) & (k : \textup{even}).
\end{cases}
\end{align}
\end{theorem}
\begin{remark}
This corresponds to the explicit formula for the observation stated at the beginning of Section 4 in \cite{K}.
\end{remark}
\begin{remark}
By combining this theorem with Theorem \ref{th:rrEFGH}, we can see that $G(k,j)$ are non-negative for all $k$.
\end{remark}
\begin{proof}

We prove $\mathbb{G}(k,j) = H(k,j)$ for even $k$ by demonstrating that when $k$ is even, $H(k,j)$ satisfies the same recurrence relation as $\mathbb{G}(k,j)$, namely
\begin{align} \nonumber
&H(k,0) = 0 \quad (k \ge 0),\\  \nonumber
&H(2,2) = 0,\quad H(2,1) = 1,\\ \label{eq:Ho}
&H(k,j) = H(k, j-1) - \sum_{m=0}^{j-2}H(k-2,m)\quad \textup{($k\ge4$:even, $j\ge1$:odd)}, \\ \label{eq:He}
&H(k,j) = H(k, j-1) + \sum_{m=j-1}^{k-2}H(k-2,m)\quad \textup{($k\ge4$:even, $j\ge2$:even)}.
\end{align}
The initial conditions $H(k,0) = 0$ $(k \ge 0)$, $H(2,2) = 0$, $H(2,1) = 1$ can be verified directly by the definition of $H(k,j)$. We will now proceed to prove \eqref{eq:Ho}. 
When $j=1$, from the recurrence relation \eqref{eq:rrH}, we have
\begin{align*}
H(k,1) = H(k,0) + (-1)^{k} H(k-1,0) + \mathbb{G}_{k-1}.
\end{align*}
Since $\mathbb{G}_{k-1} = 0$ for even $k \ge 4$, we have $H(k,1) = H(k,0) = 0$, which is equation \eqref{eq:Ho} with $j=1$. Let $j \ge 2$. By putting $k=k-1$ and $j=m+1$ in the recurrence relation \eqref{eq:rrH}, we have 
\begin{align} \label{eq:Hkm}
& (-1)^{k-1} H(k-2,m) = H(k-1,m+1) - H(k-1,m) - (-1)^{m}\mathbb{G}_{k-2}.
\end{align}
Therefore, we obtain
\begin{align*}
\sum_{m=0}^{j-2}H(k-2,m) &= (-1)^{k-1} H(k-1,j-1) - (-1)^{k-1} \delta_{j: \textup{even}}\mathbb{G}_{k-2}\\
&=H(k,j-1) - H(k,j) -  (-1)^{j}\mathbb{G}_{k-1} - (-1)^{k-1} \delta_{j: \textup{even}}\mathbb{G}_{k-2}.
\end{align*}
By setting $k \ge 4$ as even and $j \ge 3$ as odd in this equation, we obtain \eqref{eq:Ho}.

By using equation \eqref{eq:Hkm} once again, we obtain
\begin{align*}
&\sum_{m=j-1}^{k-2}H(k-2,m) \\
&= (-1)^{k-1} H(k-1,k-1) - (-1)^{k-1} H(k-1,j-1) + \delta_{k-j: \textup{odd}}\mathbb{G}_{k-2}\\
&=(-1)^{k-1} H(k-1,k-1) + H(k,j) - H(k,j-1) + (-1)^{j}\mathbb{G}_{k-1} + \delta_{k-j:odd}\mathbb{G}_{k-2}.
\end{align*}
Therefore, if we can prove $H(k-1,k-1) = 0$, then \eqref{eq:He} will be derived. Hence, our focus now shifts to the proof of $H(k-1,k-1) = 0$. By using the generating function of $G_{n}$, we obtain
\begin{align*}
2\delta_{k=1} &= G_{k} + \sum_{n=0}^{k}\binom{k}{n}G_{k-n}.
\end{align*}
This implies that 
\begin{align*}
2 \delta_{k: \textup{even}}\mathbb{G}_{k} = G(k,k)
\end{align*}
since
\begin{align*}
2\delta_{k=1} -2 G_{k} = -2 \delta_{k: \text{even}} G_{k} = 2 \delta_{k: \textup{even}} (-1)^{\lfloor (k-1)/2 \rfloor} \mathbb{G}_{k}.
\end{align*}
Moreover, by adding equations \eqref{eq:rrH} and \eqref{eq:rrHG}, we have
\begin{align*}
2H(k,j) = (-1)^{k} H(k-1,j-1) + (-1)^{k} G(k-1,j-1) - (-1)^{j}\mathbb{G}_{k-1} + \mathbb{G}_{k-1}.
\end{align*}
Therefore, we obtain
\begin{align*}
2H(k,k) &= (-1)^{k} H(k-1,k-1) + (-1)^{k} 2 \delta_{k: \textup{odd}}\mathbb{G}_{k-1} - (-1)^{k}\mathbb{G}_{k-1} + \mathbb{G}_{k-1}\\
&= (-1)^{k} H(k-1,k-1).
\end{align*}
From this recurrence relation and the initial condition $H(0,0) = 0$, it can be deduced inductively that $H(k,k) = 0$ for all $k$. Thus, \eqref{eq:He} is proven, and the proof of $\mathbb{G}(k,j) = H(k,j)$ for even $k$ is complete. Next, we will proceed to the proof of $\mathbb{G}(k,j) = G(k,j)$ for odd $k$. By putting $k=k+1$ and $j=j+1$ in the recurrence relation \eqref{eq:rrHG}, we have
\begin{align*}
H(k+1,j+1) = - H(k+1,j) + (-1)^{k+1} G(k,j) + \mathbb{G}_{k}.
\end{align*}
Therefore, when $k \ge 3$ is odd, it follows that
\begin{align*}
G(k,j) = H(k+1,j) + H(k+1,j+1) = \mathbb{G}(k+1,j) + \mathbb{G}(k+1,j+1).
\end{align*}
This value is found to be equal to $\mathbb{G}(k,j)$ from the recurrence relation satisfied by $\mathbb{G}(k,j)$. Finally, by directly verifying that $G(1,0) = G(1,1) = 0$, we can complete the proof of the theorem.
\end{proof}

\section{Relations of multiple $\tilde{T}$-values} \label{se:relTt}
The purpose of this section is to prove the main theorems (Theorem \ref{th:Tt+T=EG}, \ref{th:Tt+T=G}.  \ref{th:TorTt1=EE}, \ref{th:TorTt=E} and \ref{th:TorTt2=EE}). In Section \ref{ss:General relations}, we give general relations of multiple $\tilde{T}$-values (Theorem \ref{th: = C}, \ref{th: = D}, \ref{th: =E} and \ref{th: =F},), and from Section \ref{ss:depth one} onwards, we specialize these relations to prove the main theorems.
\subsection{Preliminaries} 

The notation $\delta_{P}$ is defined as $1$ if condition $P$ is satisfied, and $0$ otherwise.
For an index $\mathbf{k} = (k_{1},\dots,k_{r})$, $|\mathbf{k}|:=k_{1}+\dots+k_{r}$ is called the weight of $|\mathbf{k}|$ and $r$ is called the depth of $|\mathbf{k}|$. An index $\mathbf{k}$ is called admissible when $\mathbf{k} \in \mathbb{Z}_{\ge 1}^{r}$ and $k_{r} \ge 2$. Let $\emptyset$ denote the empty index and let $T(\emptyset) = \tilde{T}(\emptyset) = 1$. Let $\{k\}^{n}$ denotes $n$ repetitions of $k$. The dual index of an admissible index $\mathbf{k}$ is denoted by $\mathbf{k}^{\dag}$. In other words, 
\[\mathbf{k}^{\dag}=(\{1\}^{b_{h}-1},a_{h}+1,\dots,\{1\}^{b_{2}-1},a_{2}+1,\{1\}^{b_{1}-1},a_{1}+1),\]
when
\[\mathbf{k}=(\{1\}^{a_{1}-1},b_{1}+1,\{1\}^{a_{2}-1},b_{2}+1,\dots,\{1\}^{a_{h}-1},b_{h}+1)\quad (a_{i}, b_{i} \ge 1).\]
We put $\mathbf{k}_{j} = (k_{1}, \dots, k_{j} )$ and $\mathbf{k}^{j} = (k_{j+1}, \dots, k_{r})$ for $\mathbf{k} = (k_{1},\dots,k_{r})$.
We define $\mathbf{k}^{(n)}$ with superscript enclosed in parentheses as follows; $\mathbf{k}^{(0)}=\mathbf{k}$,  
\begin{align*}
\mathbf{k}^{(1)}=&
  \begin{cases}
    (k_{1},\dots,k_{r-1},k_{r}-1)&(k_{r}>1),\\
    (k_{1},\dots,k_{r-1})&(k_{r}=1),\\
    \phi&(r=1, k_{r}=1),
  \end{cases}
\end{align*}
and inductively $\mathbf{k}^{(n)}=(\mathbf{k}^{(n-1)})^{(1)}\ (n>1)$.

We define the multiple $A$-function as
\begin{align*}
A(\mathbf{k};z) = 2^{r}\sum_{\substack{0<m_{1}<\cdots<m_{r}\\m_{j}\equiv j\ {\rm mod}\ 2}}\frac{z^{m_{n}}}{m^{k_{1}}_{1}\cdots m^{k_{r}}_{r}}.
\end{align*}
It can be seen from the definitions that $A(\mathbf{k};1) = {T}(\mathbf{k})$ and $A(\mathbf{k};i) = i^{r}\tilde{T}(\mathbf{k})$. The multiple $A$-function has the following iterated integral representation:
\begin{align*}
A(\mathbf{k};z) = \int_{0}^{z} \Omega_{1} \circ \Omega_{2} \circ \cdots \circ \Omega_{|\mathbf{k}|} = \int_{0}^{z}\cdots \left(\int_{0}^{t_{3}}\left(\int_{0}^{t_{2}} \Omega_{1}\right) \Omega_{2}\right) \cdots \Omega_{|\mathbf{k}|}, 
\end{align*}
where 
\begin{align*}
\Omega_{i} = \begin{dcases}
\frac{2dt_{i}}{1-t_{i}^2}& (i \in \{1, k_{1}+1, \dots, k_{1}+\cdots+k_{r-1}+1\}),\\
\frac{dt_{i}}{t_{i}}&\textup{(otherwise)}.
\end{dcases}
\end{align*}
Moreover, from this expression, multiple $A$-functions satisfy the shuffle product formula.

The following facts stated in \cite{KT3} will be used in our calculations later:
\begin{align}\label{eq:TorTt=E}
\begin{rcases}
{T}({k+1}) & (k:\textup{odd})\\
\tilde{T}({k+1}) & (k:\textup{even})
\end{rcases}=\frac{\mathbb{E}_{k}}{k!} \left(\frac{\pi}{2}\right)^{k+1},\\ \nonumber
\tilde{T}(\{1\}^{n}) =  \tilde{T}(1)^{n}/n! = (\pi/2)^{n}/n!.
\end{align}
The formula \eqref{eq:TorTt=E} will be derived again from our calculations.

\subsection{General relations} \label{ss:General relations}
We first introduce the following notation.
\begin{definition} \label{def;calT}
Let $\{a_{n}\}_{n = 0}^{\infty}$ be any sequence.
For $h \in \mathbb{Z}_{\ge 0}$ and admissible index $\mathbf{k} = (k_{1}, \dots, k_{r})$, we define
\begin{align*}
\mathcal{T}(a_{n};\mathbf{k};h) = \sum_{n=0}^{k_{r}-2} \frac{a_{n}}{n!} \left(\frac{\pi}{2}\right)^{n}\tilde{T}((\mathbf{k}^{(n)})^{\dag},\{1\}^{h}).
\end{align*}
\end{definition}
Then, the general relations to be proved can be stated as follows.
\begin{theorem} \label{th: = C} For $k_{r} \ge 2$, we have
\begin{align*}
& i^{r+1} \frac{\pi}{2} \tilde{T}(\mathbf{k}^{(1)}) + i^{r}\frac{(-1)^{(k_{r}-1)/2}{C}_{k_{r}-1}}{(k_{r}-1)!} \left(\frac{\pi}{2}\right)^{k_{r}-1} \tilde{T}(\mathbf{k}^{(k_{r}-1)}) \\
& + \sum_{\substack{j=0\\j:{\rm even}}}^{k_{r}-2}\frac{(-1)^{j/2}{C}_{j}}{j!}\left(\frac{\pi}{2}\right)^{j} {T}(\mathbf{k}^{(j)}) - i^{r} \sum_{\substack{j=0\\j:{\rm even}}}^{k_{r}-1}\frac{(-1)^{j/2}{D}_{j}}{j!}\left(\frac{\pi}{2}\right)^{j} \tilde{T}(\mathbf{k}^{(j)}) \\
& = i^{r-|\mathbf{k}|} \mathcal{T}({C}_{n};\mathbf{k};0) + \sum_{j=1}^{r-1}\sum_{h=0}^{k_{j}-1} i^{r-|\mathbf{k}^{j}|-h} \tilde{T}((\mathbf{k}_{j})^{(h)}) \mathcal{T}({C}_{n};\mathbf{k}^{j};h).
\end{align*}
\end{theorem}
\begin{theorem} \label{th: = D} For $k_{r} \ge 3$, we have
\begin{align*}
& i \frac{\pi}{2} {T}(\mathbf{k}^{(1)}) + i^{r}\frac{(-1)^{(k_{r}-1)/2} {D}_{k_{r}-1}}{(k_{r}-1)!} \left(\frac{\pi}{2}\right)^{k_{r}-1} \tilde{T}(\mathbf{k}^{(k_{r}-1)}) \\
& + \sum_{\substack{j=0\\j:{\rm even}}}^{k_{r}-2}\frac{(-1)^{j/2}{D}_{j}}{j!}\left(\frac{\pi}{2}\right)^{j} {T}(\mathbf{k}^{(j)}) - i^{r} \sum_{\substack{j=0\\j:{\rm even}}}^{k_{r}-1}\frac{(-1)^{j/2}{C}_{j}}{j!}\left(\frac{\pi}{2}\right)^{j} \tilde{T}(\mathbf{k}^{(j)})  \\
& = i^{r-|\mathbf{k}|} \mathcal{T}({D}_{n};\mathbf{k};0) + \sum_{j=1}^{r-1}\sum_{h=0}^{k_{j}-1} i^{r-|\mathbf{k}^{j}|-h} \tilde{T}((\mathbf{k}_{j})^{(h)}) \mathcal{T}({D}_{n};\mathbf{k}^{j};h).
\end{align*}
\end{theorem}
\begin{theorem} \label{th: =E} For $k_{r} \ge 2$, we have
\begin{align*}
&- i^{r} \tilde{T}(\mathbf{k}) + i^{r}\delta_{k_{r}: \textup{odd}}\frac{\mathbb{E}_{k_{r}-1}}{(k_{r}-1)!} \left(\frac{\pi}{2}\right)^{k_{r}-1} \tilde{T}(\mathbf{k}^{(k_{r}-1)})\\
&+\sum_{\substack{j=0\\j:{\rm even}}}^{k_{r}-2}\frac{\mathbb{E}_{j}}{j!}\left(\frac{\pi}{2}\right)^{j}T(\mathbf{k}^{(j)}) + i^{r+1}\sum_{\substack{j=0\\j:{\rm odd}}}^{k_{r}-1}\frac{\mathbb{E}_{j}}{j!} \left(\frac{\pi}{2}\right)^{j}\tilde{T}(\mathbf{k}^{(j)}) \\
&= i^{r-|\mathbf{k}|} \mathcal{T}(E_{n};\mathbf{k};0) + \sum_{j=1}^{r-1}\sum_{h=0}^{k_{j}-1} i^{r-|\mathbf{k}^{j}|-h} \tilde{T}((\mathbf{k}_{j})^{(h)}) \mathcal{T}(E_{n};\mathbf{k}^{j};h) .
\end{align*}
\end{theorem}
\begin{theorem} \label{th: =F} For $k_{r} \ge 2$, we have
\begin{align*}
&T(\mathbf{k}) + i^{r+1} \delta_{k_{r}: \textup{even}} \frac{\mathbb{E}_{k_{r}-1}}{(k_{r}-1)!} \left(\frac{\pi}{2}\right)^{k_{r}-1} \tilde{T}(\mathbf{k}^{(k_{r}-1)})\\
&+i \sum_{\substack{j=0\\j:{\rm odd}}}^{k_{r}-2}\frac{\mathbb{E}_{j}}{j!} \left(\frac{\pi}{2}\right)^{j} {T}(\mathbf{k}^{(j)}) -i^{r}\sum_{\substack{j=0\\j:{\rm even}}}^{k_{r}-1} \frac{\mathbb{E}_{j}}{j!} \left(\frac{\pi}{2}\right)^{j} \tilde{T}(\mathbf{k}^{(j)})\\
&= i^{r-|\mathbf{k}|} \mathcal{T}(F_{n};\mathbf{k};0)  + \sum_{j=1}^{r-1}\sum_{h=0}^{k_{j}-1} i^{r-|\mathbf{k}^{j}|-h} \tilde{T}((\mathbf{k}_{j})^{(h)}) \mathcal{T}(F_{n};\mathbf{k}^{j};h) .
\end{align*}
\end{theorem}
Theorem \ref{th: = C} and Theorem \ref{th: = D} yield the following corollary.
\begin{corollary} \label{col: =G} For $k_{r} \ge 3$, we have
\begin{align*}
& i^{r+1} \frac{\pi}{2} \tilde{T}(\mathbf{k}^{(1)}) - i \frac{\pi}{2} {T}(\mathbf{k}^{(1)}) + i^{r}\frac{\mathbb{G}_{k_{r}-1}}{(k_{r}-1)!} \left(\frac{\pi}{2}\right)^{k_{r}-1} \tilde{T}(\mathbf{k}^{(k_{r}-1)}) \\
& + \sum_{\substack{j=2\\j:{\rm even}}}^{k_{r}-2}\frac{\mathbb{G}_{j}}{j!}\left(\frac{\pi}{2}\right)^{j} {T}(\mathbf{k}^{(j)}) + i^{r} \sum_{\substack{j=2\\j:{\rm even}}}^{k_{r}-1}\frac{\mathbb{G}_{j}}{j!}\left(\frac{\pi}{2}\right)^{j} \tilde{T}(\mathbf{k}^{(j)})  \\
& = i^{r-|\mathbf{k}|} \mathcal{T}({G}_{n};\mathbf{k};0) + \sum_{j=1}^{r-1}\sum_{h=0}^{k_{j}-1} i^{r-|\mathbf{k}^{j}|-h} \tilde{T}((\mathbf{k}_{j})^{(h)}) \mathcal{T}({G}_{n};\mathbf{k}^{j};h).
\end{align*}
\end{corollary}

The proofs of these theorems start from the following theorem.

\begin{theorem}[\cite{U3}] For an admissible index $\mathbf{k}$, we have
\begin{align*}
T(\mathbf{k}) = \sum_{h=0}^{|\mathbf{k}|} A(\mathbf{k}^{(h)};i)\overline{A((\mathbf{k}^{\dag})^{(|\mathbf{k}|-h)};i)}. 
\end{align*}
\end{theorem}
In rewriting the right-hand side, we obtain
\begin{align}\nonumber
T(\mathbf{k}) &= \sum_{h=0}^{k_{r}-1} A(\mathbf{k}^{(h)};i)\overline{A(\{1\}^{h};i)}\\ \nonumber
&\quad + \overline{A(\mathbf{k}^{\dag};i)} + \sum_{j=1}^{r-1}\sum_{h=0}^{k_{j}-1} A((\mathbf{k}_{j})^{(h)};i)\overline{A((\mathbf{k}^{j})^{\dag},\{1\}^{h};i)}\\ \label{eq:T=}
&=  i^{r}\sum_{h=0}^{k_{r}-1}\tilde{T}(\mathbf{k}^{(h)}) \frac{(-i \pi/2)^{h}}{h!}\\ \nonumber
&\quad + i^{r-|\mathbf{k}|}\tilde{T}(\mathbf{k}^{\dag}) + \sum_{j=1}^{r-1}\sum_{h=0}^{k_{j}-1} i^{r-|\mathbf{k}^{j}|-h} \tilde{T}((\mathbf{k}_{j})^{(h)})\tilde{T}((\mathbf{k}^{j})^{\dag},\{1\}^{h})
\end{align}
By multiplying this equation by $X^{k_{r}}$ and summing over $k_{r}\ge2$, we obtain the following formal power series relation:
\begin{align}
\label{eq: fps}
\sum_{k_{r}=2}^{\infty}T(\mathbf{k})X^{k_{r}} &= - i^{r}\tilde{T}(\mathbf{k}^{(k_{r}-1)})X + i^{r}e^{-i(\pi/2)X}\sum_{k_{r}=1}^{\infty} \tilde{T}(\mathbf{k}) X^{k_{r}} \\ \nonumber
&\quad + i^{r} \sum_{j=1}^{r-1}\sum_{h=0}^{k_{j}-1} \tilde{T}((\mathbf{k}_{j})^{(h)}) \sum_{k_{r}=2}^{\infty} i^{-|\mathbf{k}^{j}|-h}\tilde{T}((\mathbf{k}^{j})^{\dag},\{1\}^{h})X^{k_{r}}\\ \nonumber
&\quad + i^{r} \sum_{k_{r}=2}^{\infty} i^{-|\mathbf{k}|}\tilde{T}(\mathbf{k}^{\dag}) X^{k_{r}}. 
\end{align}
Here, the calculation of the first member of the right-hand side of \eqref{eq:T=} is performed as follows:
\begin{align*}
&\sum_{k_{r}=2}^{\infty} \sum_{h=0}^{k_{r}-1} \tilde{T}(\mathbf{k}^{(h)}) \frac{(-i \pi/2)^{h}}{h!}X^{k_{r}}\\
&=\sum_{k_{r}=2}^{\infty} \sum_{h=0}^{k_{r}-1} \tilde{T}(k_{1},\dots,k_{r-1},k_{r}-h) \frac{(-i \pi/2)^{h}}{h!}X^{k_{r}}\\
&=\sum_{k_{r}=1}^{\infty} \sum_{h=0}^{k_{r}-1} \tilde{T}(k_{1},\dots,k_{r-1},k_{r}-h) \frac{(-i \pi/2)^{h}}{h!}X^{k_{r}}\\
&\quad - \tilde{T}(k_{1},\dots,k_{r-1},1) X\\
&=- \tilde{T}(\mathbf{k}^{(k_{r}-1)})X + e^{-i(\pi/2)X} \sum_{k_{r}=1}^{\infty} \tilde{T}(\mathbf{k})X^{k_{r}}.
\end{align*}
Theorems \ref{th: = C}, \ref{th: = D}, \ref{th: =E} and \ref{th: =F} are obtained by multiplying the both sides of equation \eqref{eq: fps} by $\sum_{n=0}^{\infty} a_{n}\left(-i(\pi/2) X\right)^{n}/n!$ and comparing coefficients. First, note that $\mathcal{T}(a_{n};\mathbf{k};h)$ satisfies
\begin{align*}
&\sum_{n=0}^{\infty}\frac{a_{n}}{n!}\left(-\frac{\pi}{2}i X\right)^{n}\cdot \sum_{k_{r}=2}^{\infty}i^{-|\mathbf{k}|-h}\tilde{T}(\mathbf{k}^{\dag},\{1\}^{h}) X^{k_{r}} = \sum_{k_{r}=2}^{\infty} i^{-|\mathbf{k}|-h} \mathcal{T}(a_{n};\mathbf{k};h) X^{k_{r}}.
\end{align*}

To obtain Thorem \ref{th: = C}, we multiply the both sides of equation \eqref{eq: fps} by
\begin{align*}
\sum_{n=0}^{\infty} \frac{{C}_{n}}{n!}\left(-\frac{\pi}{2}i X\right)^{n} = \sum_{\substack{n=0\\ n:{\rm even}}}^{\infty} \frac{(-1)^{n/2}{C}_{n}}{n!}\left(\frac{\pi}{2} X\right)^{n}.
\end{align*}
Then, noting the equation
\begin{align*}
e^{-i(\pi/2) X}\sum_{n=0}^{\infty} \frac{{C}_{n}}{n!}\left(-\frac{\pi}{2}i X\right)^{n} &= \sum_{n=0}^{\infty} \frac{{D}_{n}}{n!}\left(\frac{\pi}{2} iX\right)^{n} \\
&= -\frac{\pi}{2} i X + \sum_{\substack{n=0\\ n:{\rm even}}}^{\infty} \frac{(-1)^{n/2}{D}_{n}}{n!}\left(\frac{\pi}{2} X\right)^{n},
\end{align*}
we can obtain Theorem \ref{th: = C} by comparing the coefficients of $X^{k_{r}}$ with $k_{r} \ge 2$. In the same way, Theorem \ref{th: = D} can be obtained by multiplying equation \eqref{eq: fps} by 
\begin{align*}
\sum_{n=0}^{\infty} \frac{{D}_{n}}{n!}\left(-\frac{\pi}{2}i X\right)^{n} &= \frac{\pi}{2} i X  +\sum_{\substack{n=0\\ n:{\rm even}}}^{\infty} \frac{(-1)^{n/2}{D}_{n}}{ n!}\left(\frac{\pi}{2} X\right)^{n}
\end{align*}
and noting
\begin{align*}
e^{-i(\pi/2) X}\sum_{n=0}^{\infty} \frac{{D}_{n}}{n!}\left(-\frac{\pi}{2}i X\right)^{n} &= \sum_{n=0}^{\infty} \frac{{C}_{n}}{n!}\left(\frac{\pi}{2}i X\right)^{n}= \sum_{\substack{n=0\\ n:{\rm even}}}^{\infty} \frac{(-1)^{n/2}{C}_{n}}{n!}\left(\frac{\pi}{2} X\right)^{n}.
\end{align*}
Theorem \ref{th: =E} can be obtained by multiplying equation \eqref{eq: fps} by
\begin{align*}
\sum_{n=0}^{\infty} \frac{E_{n}}{n!}\left(-\frac{\pi}{2}i X\right)^{n} = \sum_{\substack{n=0\\n:{\rm even}}}^{\infty} \frac{\mathbb{E}_{n}}{n!}\left(\frac{\pi}{2} X\right)^{n}
\end{align*}
and  noting
\begin{align*}
e^{-i(\pi/2) X}\sum_{n=0}^{\infty} \frac{E_{n}}{n!}\left(-\frac{\pi}{2}i X\right)^{n} &= \sum_{n=0}^{\infty} \frac{{F}_{n}}{n!}\left(\frac{\pi}{2}i X\right)^{n} = 1-i\sum_{\substack{n=0\\ n:{\rm odd}}}^{\infty}\frac{\mathbb{E}_{n}}{n!} \left(\frac{\pi}{2}X\right)^{n}.
\end{align*}
Theorem \ref{th: =F} can be obtained by multiplying equation \eqref{eq: fps} by 
\begin{align*}
\sum_{n=0}^{\infty} \frac{F_{n}}{n!}\left(-\frac{\pi}{2}i X\right)^{n} = 1 + i \sum_{\substack{n=0\\ n:{\rm odd}}}^{\infty} \frac{\mathbb{E}_{n}}{n!}\left(\frac{\pi}{2} X\right)^{n}
\end{align*}
and noting
\begin{align*}
e^{-i(\pi/2) X}\sum_{n=0}^{\infty} \frac{F_{n}}{n!}\left(-\frac{\pi}{2}i X\right)^{n} = \sum_{n=0}^{\infty} \frac{E_{n}}{n!}\left(\frac{\pi}{2}i X\right)^{n} =  \sum_{\substack{n=0\\ n:{\rm even}}}^{\infty} \frac{\mathbb{E}_{n}}{n!}\left(\frac{\pi}{2} X\right)^{n}.
\end{align*}

\subsection{Lemmas}
We will give two lemmas in order to specialize Theorem \ref{th: = C}, \ref{th: = D}, \ref{th: =E} and \ref{th: =F}.
By the shuffle product formula, $\mathcal{T}(a_{n};\mathbf{k};h)$ can also be written as follows.
\begin{lemma} \label{lem;calT=CT}We have
\begin{align*}
&\mathcal{T}(a_{n};\mathbf{k};h) \\
&= \sum_{\substack{m_{0}+|\mathbf{m}| = k_{r}-1 \\ m_{r} > 0}} \binom{h+m_{0}}{m_{0}}  \left(\prod_{u=1}^{r-1}\binom{k_{u}+m_{u}-1}{m_{u}}\right) \left(\sum_{n=1}^{m_{r}} \binom{m_{r}}{n} a_{k_{r}-1-n} \right) \\
&\quad \times \tilde{T}((\mathbf{k}^{(k_{r}-1)} + \mathbf{m})^{\dag},\{1\}^{h+m_{0}}),
\end{align*}
where the sum is over all $m_{0} \in \mathbb{Z}_{\ge 0}$ and $\mathbf{m} = (m_{1}, \dots, m_{r}) \in \mathbb{Z}_{\ge 0}^{r}$ satisfying $m_{0}+|\mathbf{m}| = k_{r}-1$ and $m_{r} > 0$.
\end{lemma}
\begin{proof}
By the formula $\tilde{T}(\{1\}^{n}) = (\pi/2)^{n}/n!$ and the shuffle product formula, we have
\begin{align*}
&\mathcal{T}(a_{n};\mathbf{k};h) \\
&= \sum_{n=0}^{k_{r}-2} a_{n} \tilde{T}(\{1\}^{n}) \tilde{T}((\mathbf{k}^{(n)})^{\dag},\{1\}^{h})\\
&= \sum_{n=2}^{k_{r}} a_{k_{r}-n} \tilde{T}(\{1\}^{k_{r}-n}) \tilde{T}((\mathbf{k}^{(k_{r}-n)})^{\dag},\{1\}^{h})\\
&= \sum_{n=2}^{k_{r}} a_{k_{r}-n} \sum_{\substack{m_{0}+|\mathbf{m}| = k_{r}-n}} \binom{h+m_{0}}{m_{0}} \left(\prod_{u=1}^{r-1}\binom{k_{u}+m_{u}-1}{m_{u}}\right) \binom{n+m_{r}-1}{m_{r}} \\
&\quad \times \tilde{T}((\mathbf{k}^{(k_{r}-n)} + \mathbf{m})^{\dag},\{1\}^{h+m_{0}})\\
&= \sum_{\substack{m_{0}+|\mathbf{m}| = k_{r}-1 \\ m_{r} > 0}} \sum_{n=2}^{m_{r}+1} a_{k_{r}-n} \binom{h+m_{0}}{m_{0}} \left(\prod_{u=1}^{r-1}\binom{k_{u}+m_{u}-1}{m_{u}}\right) \binom{m_{r}}{n  - 1} \\
&\quad \times \tilde{T}((\mathbf{k}^{(k_{r}-1)} + \mathbf{m})^{\dag},\{1\}^{h+m_{0}})\\
&= \sum_{\substack{m_{0}+|\mathbf{m}| = k_{r}-1 \\ m_{r} > 0}} \binom{h+m_{0}}{m_{0}}  \left(\prod_{u=1}^{r-1}\binom{k_{u}+m_{u}-1}{m_{u}}\right) \left(\sum_{n=1}^{m_{r}} \binom{m_{r}}{n} a_{k_{r}-1-n} \right) \\
&\quad \times \tilde{T}((\mathbf{k}^{(k_{r}-1)} + \mathbf{m})^{\dag},\{1\}^{h+m_{0}}).
\end{align*}
The second equality from the last is obtained by putting $m_{r} + n - 1$ as $m_{r}$.
\end{proof}
The following lemma is also useful in later calculations.
\begin{lemma} \label{lem; TtcalT}
For a admissible $\mathbf{k} = (k_{1}, \dots, k_{r})$ with $r \ge 2$, we have
\begin{align*}
&\begin{dcases}
\Re\left(\sum_{h=0}^{k_{1}-1} i^{r-|\mathbf{k}^{1}|-h} \tilde{T}((\mathbf{k}_{1})^{(h)}) \mathcal{T}(a_{n};\mathbf{k}^{1};h) \right) & (|\mathbf{k}|-r : \textup{odd}),\\
i\Im\left(\sum_{h=0}^{k_{1}-1} i^{r-|\mathbf{k}^{1}|-h} \tilde{T}((\mathbf{k}_{1})^{(h)}) \mathcal{T}(a_{n};\mathbf{k}^{1};h) \right) & (|\mathbf{k}|-r : \textup{even})
\end{dcases}\\
&=i^{r-|\mathbf{k}|+1} \sum_{m=0}^{k_{1}-1} \sum_{n=0}^{k_{r}-2} \frac{E_{m}}{m!} \frac{a_{n}}{n!}\left(\frac{\pi}{2}\right)^{m+n+1}\tilde{T}(((\mathbf{k}^{1})^{(n)})^{\dag},\{1\}^{k_{1}-1-m})\\
&=i^{r-|\mathbf{k}|+1} \frac{\pi}{2} \sum_{\substack{|\mathbf{l}| = k_{1}-1}}  \sum_{\substack{|\mathbf{m}| = k_{r}-1 \\ m_{r} > 0}} \left( \sum_{h=0}^{l_{1}} \binom{l_{1}}{h} {E}_{k_{1}-1-h} \right) \left(\sum_{n=1}^{m_{r}} \binom{m_{r}}{n} a_{k_{r}-1-n} \right) \\
&\quad \times \binom{l_{1}+m_{1}}{l_{1}} \left(\prod_{u=2}^{r-1}\binom{k_{u}-1+l_{u}+m_{u}}{k_{u}-1,l_{u},m_{u}}\right)  \binom{l_{r}+m_{r}}{l_{r}} \\
&\quad \times \tilde{T}(((\mathbf{k}^{1})^{(k_{r}-1)} + \mathbf{l}^{1} + \mathbf{m}^{1})^{\dag},\{1\}^{l_{1} + m_{1}}).
\end{align*}
\end{lemma}
\begin{proof} 
The first equality is obtained by \eqref{eq:TorTt=E} as follows:
\begin{align*}
&i^{r-|\mathbf{k}|+1} \sum_{m=0}^{k_{1}-1} \sum_{n=0}^{k_{r}-2} \frac{E_{m}}{m!} \frac{a_{n}}{n!}\left(\frac{\pi}{2}\right)^{m+n+1}\tilde{T}(((\mathbf{k}^{1})^{(n)})^{\dag},\{1\}^{k_{1}-1-m})\\
&=i^{r-|\mathbf{k}|+1} \sum_{\substack{m=0 \\ m:{\rm even}}}^{k_{1}-1} \frac{E_{m}}{m!} \left(\frac{\pi}{2}\right)^{m+1}\sum_{n=0}^{k_{r}-2} \frac{a_{n}}{n!} \left(\frac{\pi}{2}\right)^{n}\tilde{T}(((\mathbf{k}^{1})^{(n)})^{\dag},\{1\}^{k_{1}-1-m})\\
&=i^{r-|\mathbf{k}|+1} \sum_{\substack{m=0 \\ m:{\rm even}}}^{k_{1}-1} i^{m} \tilde{T}(m+1) \mathcal{T}(a_{n};\mathbf{k}^{1};k_{1}-1-m)\\
&=i^{r-|\mathbf{k}|+1} \sum_{\substack{h=0 \\ k_{1}-1-h:{\rm even}}}^{k_{1}-1} i^{k_{1}-1-h} \tilde{T}(k_{1}-h) \mathcal{T}(a_{n};\mathbf{k}^{1};h)\\
&=
\begin{dcases}
\Re\left(\sum_{h=0}^{k_{1}-1} i^{r-|\mathbf{k}^{1}|-h} \tilde{T}((\mathbf{k}_{1})^{(h)}) \mathcal{T}(a_{n};\mathbf{k}^{1};h) \right) & (|\mathbf{k}|-r : \textup{odd}),\\
i\Im\left(\sum_{h=0}^{k_{1}-1} i^{r-|\mathbf{k}^{1}|-h} \tilde{T}((\mathbf{k}_{1})^{(h)}) \mathcal{T}(a_{n};\mathbf{k}^{1};h) \right) & (|\mathbf{k}|-r : \textup{even}).
\end{dcases}
\end{align*}
The second equality in the lemma is obtained by the shuffle product formula as follows:
\begin{align*}
&i^{r-|\mathbf{k}|+1} \sum_{\substack{h=0 \\ k_{1}-1-h:{\rm even}}}^{k_{1}-1} i^{k_{1}-1-h} \tilde{T}(k_{1}-h) \mathcal{T}(a_{n};\mathbf{k}^{1};h)\\
&=i^{r-|\mathbf{k}|+1} \frac{\pi}{2}\sum_{\substack{h=0 \\ k_{1}-1-h:{\rm even}}}^{k_{1}-1} {E}_{k_{1}-1-h} \tilde{T}(\{1\}^{k_{1}-1-h}) \mathcal{T}(a_{n};\mathbf{k}^{1};h)\\
&=i^{r-|\mathbf{k}|+1} \frac{\pi}{2}\sum_{\substack{h=0 \\ k_{1}-1-h:{\rm even}}}^{k_{1}-1} \sum_{\substack{m_{0}+|\mathbf{m}^{1}| = k_{r}-1 \\ m_{r} > 0}} {E}_{k_{1}-1-h}\\
&\quad \times \binom{h+m_{0}}{m_{0}} \left(\prod_{u=2}^{r-1}\binom{k_{u}+m_{u}-1}{m_{u}}\right)  \left(\sum_{n=1}^{m_{r}} \binom{m_{r}}{n} a_{k_{r}-1-n} \right)\\
&\quad \times\tilde{T}(\{1\}^{k_{1}-1-h}) \tilde{T}(((\mathbf{k}^{1})^{(k_{r}-1)} + \mathbf{m}^{1})^{\dag},\{1\}^{h+m_{0}})\\
&=i^{r-|\mathbf{k}|+1} \frac{\pi}{2}\sum_{\substack{h=0 \\ k_{1}-1-h:{\rm even}}}^{k_{1}-1} \sum_{\substack{m_{0}+|\mathbf{m}^{1}| = k_{r}-1 \\ m_{r} > 0}} {E}_{k_{1}-1-h}\\
&\quad \times \binom{h+m_{0}}{m_{0}} \left(\prod_{u=2}^{r-1}\binom{k_{u}+m_{u}-1}{m_{u}}\right)  \left(\sum_{n=1}^{m_{r}} \binom{m_{r}}{n} a_{k_{r}-1-n} \right)\\
&\quad \times \sum_{\substack{l_{0}+|\mathbf{l}^{1}| = k_{1}-1-h}} \binom{h+l_{0} +m_{0}}{l_{0}} \left(\prod_{u=2}^{r-1}\binom{k_{u}+l_{u}+m_{u}-1}{l_{u}}\right) \binom{l_{r}+m_{r}}{l_{r}} \\
&\quad \times \tilde{T}(((\mathbf{k}^{1})^{(k_{r}-1)} + \mathbf{l}^{1} + \mathbf{m}^{1})^{\dag},\{1\}^{h+l_{0} +m_{0}})\\
&=i^{r-|\mathbf{k}|+1} \frac{\pi}{2} \sum_{\substack{l_{0}+|\mathbf{l}^{1}| = k_{1}-1}}  \sum_{\substack{m_{0}+|\mathbf{m}^{1}| = k_{r}-1 \\ m_{r} > 0}} \sum_{\substack{h=0 \\ k_{1}-1-h:{\rm even}}}^{l_{0}}  {E}_{k_{1}-1-h}\\
&\quad \times \binom{h+m_{0}}{m_{0}} \left(\prod_{u=2}^{r-1}\binom{k_{u}+m_{u}-1}{m_{u}}\right)  \left(\sum_{n=1}^{m_{r}} \binom{m_{r}}{n} a_{k_{r}-1-n} \right)\\
&\quad \times \binom{l_{0} +m_{0}}{l_{0}-h} \left(\prod_{u=2}^{r-1}\binom{k_{u}+l_{u}+m_{u}-1}{l_{u}}\right) \binom{l_{r}+m_{r}}{l_{r}} \\
&\quad \times \tilde{T}(((\mathbf{k}^{1})^{(k_{r}-1)} + \mathbf{l}^{1} + \mathbf{m}^{1})^{\dag},\{1\}^{l_{0} +m_{0}})\\
&=i^{r-|\mathbf{k}|+1} \frac{\pi}{2} \sum_{\substack{l_{0}+|\mathbf{l}^{1}| = k_{1}-1}}  \sum_{\substack{m_{0}+|\mathbf{m}^{1}| = k_{r}-1 \\ m_{r} > 0}} \left( \sum_{h=0}^{l_{0}} \binom{l_{0}}{h} {E}_{k_{1}-1-h} \right) \\
&\quad \times \binom{l_{0}+m_{0}}{l_{0}} \left(\prod_{u=2}^{r-1}\binom{k_{u}-1+l_{u}+m_{u}}{k_{u}-1,l_{u},m_{u}}\right)  \binom{l_{r}+m_{r}}{l_{r}} \left(\sum_{n=1}^{m_{r}} \binom{m_{r}}{n} a_{k_{r}-1-n} \right) \\
&\quad \times \tilde{T}(((\mathbf{k}^{1})^{(k_{r}-1)} + \mathbf{l}^{1} + \mathbf{m}^{1})^{\dag},\{1\}^{l_{0} + m_{0}}).
\end{align*}
The last equality is obtained by 
\begin{align*}
&\binom{h+m_{0}}{m_{0}} \binom{l_{0}+m_{0}}{l_{0}-h} = \binom{l_{0}}{h}\binom{l_{0}+m_{0}}{l_{0}}
\end{align*}
and
\begin{align*}
&\left(\prod_{u=2}^{r-1}\binom{k_{u}+m_{u}-1}{m_{u}}\right)\left(\prod_{u=2}^{r-1}\binom{k_{u}+l_{u}+m_{u}-1}{l_{u}}\right) = \prod_{u=2}^{r-1}\binom{k_{u}-1+l_{u}+m_{u}}{k_{u}-1,l_{u},m_{u}}.
\end{align*}
Finally, the lemma is proved by replacing $l_{0}$ and $m_{0}$ with $l_{1}$ and $m_{1}$ respectively.
\end{proof}

\subsection{Depth one} \label{ss:depth one}
In this subsection, we consider the case $r = 1$. In the case $\mathbf{k} = (k+1)$ with $k \ge 1$, Theorem \ref{th: =E} and Theorem \ref{th: =F} reduce to
\begin{align} \label{eq: Tt r=1}
&- i \tilde{T}({k+1}) + i \delta_{k: \textup{even}}\frac{\mathbb{E}_{k}}{k!} \left(\frac{\pi}{2}\right)^{k} \tilde{T}({1})\\ \nonumber
&+\sum_{\substack{j=0\\j:{\rm even}}}^{k-1}\frac{\mathbb{E}_{j}}{j!}\left(\frac{\pi}{2}\right)^{j}T({k+1-j}) - \sum_{\substack{j=0\\j:{\rm odd}}}^{k}\frac{\mathbb{E}_{j}}{j!} \left(\frac{\pi}{2}\right)^{j}\tilde{T}({k+1-j}) \\ \nonumber
&= i^{-{k}} \mathcal{T}(E_{n};{k+1};0)
\end{align}
and
\begin{align} \label{eq: T r=1}
&T({k+1}) - \delta_{k: \textup{odd}} \frac{\mathbb{E}_{k}}{k!} \left(\frac{\pi}{2}\right)^{k} \tilde{T}(1)\\ \nonumber
&+i \sum_{\substack{j=0\\j:{\rm odd}}}^{k-1}\frac{\mathbb{E}_{j}}{j!} \left(\frac{\pi}{2}\right)^{j} {T}({k+1-j}) - i \sum_{\substack{j=0\\j:{\rm even}}}^{k} \frac{\mathbb{E}_{j}}{j!} \left(\frac{\pi}{2}\right)^{j} \tilde{T}({k+1-j})\\ \nonumber
&= i^{-k} \mathcal{T}(F_{n};{k+1};0),
\end{align}
respectively. By taking the imaginary part of \eqref{eq: Tt r=1} and taking the real part of \eqref{eq: T r=1}, we have
\begin{align} \label{eq:Ttk_1}
\tilde{T}({k+1})
&=\begin{dcases} i^{1-k} \mathcal{T}(E_{n};{k+1};0) & (k:{\rm odd}),\\
\frac{\mathbb{E}_{k}}{k!} \left(\frac{\pi}{2}\right)^{k+1} & (k:{\rm even})\end{dcases}
\end{align}
and
\begin{align} \label{eq:Tk_1}
T({k+1})
&=\begin{dcases} \frac{\mathbb{E}_{k}}{k!} \left(\frac{\pi}{2}\right)^{k+1} & (k:{\rm odd}),\\
i^{-k} \mathcal{T}(F_{n};{k+1};0) & (k:{\rm even}). \end{dcases}
\end{align}
Equation \eqref{eq: Tt r=1} holds only for $k \ge 1$, but since $\tilde{T}(1) = \pi/2$, equation \eqref{eq:Ttk_1} also holds for $k = 0$. These equations are equivalent to \eqref{eq:TorTt=E} and the Kaneko-Tsumura relation \eqref{eq:KTrel} because Lemma \ref{lem;calT=CT} gives
\begin{align*}
\mathcal{T}(a_{n};{k+1};0) &= \sum_{\substack{m_{0}+m_{1} = k \\ m_{1} > 0}} \left(\sum_{n=1}^{m_{1}} \binom{m_{1}}{n} a_{k-n} \right) \tilde{T}(\{1\}^{m_{1}-1},2,\{1\}^{m_{0}}),
\end{align*}
and Theorem \ref{th:BE=} gives
\begin{align*}
\left(\sum_{n=1}^{m_{1}} \binom{m_{1} }{n} a_{k-n} \right) = \begin{cases}(-1)^{(k-1)/2}\mathbb{E}(k,m_{1}) & \textup{($a_{n} = E_{n}$ and $k$ is odd)}, \\
(-1)^{k/2}\mathbb{E}(k,m_{1}) & \textup{($a_{n} = F_{n}$ and $k$ is even)}. \end{cases}
\end{align*}
\subsection{Depth two}
In this subsection, we consider the case $r = 2$. In the case $\mathbf{k} = (k_{1}+1, k_{2}+1)$ with $k_{1} \ge 0$, Theorem \ref{th: = C} and Theorem \ref{th: = D} reduce to
\begin{align} \label{eq: Tt r=2}
& - i \frac{\pi}{2} \tilde{T}(k_{1}+1, k_{2}) - \frac{(-1)^{k_{2}/2} {C}_{k_{2}}}{k_{2}!} \left(\frac{\pi}{2}\right)^{k_{2}} \tilde{T}(k_{1}+1, 1) \\ \nonumber
& + \sum_{\substack{j=0\\j:{\rm even}}}^{k_{2}-1}\frac{(-1)^{j/2}{C}_{j}}{j!}\left(\frac{\pi}{2}\right)^{j} {T}(k_{1}+1, k_{2}+1-j) \\ \nonumber
&+ \sum_{\substack{j=0\\j:{\rm even}}}^{k_{2}}\frac{(-1)^{j/2}{D}_{j}}{j!}\left(\frac{\pi}{2}\right)^{j} \tilde{T}(k_{1}+1, k_{2}+1-j) \\ \nonumber
& = i^{-k_{1}-k_{2}} \mathcal{T}({C}_{n};k_{1}+1, k_{2}+1;0) + \sum_{h=0}^{k_{1}} i^{1-{k}_{2}-h} \tilde{T}({k}_{1}+1 - h) \mathcal{T}({C}_{n};{k}_{2}+1;h)
\end{align}
for $k_{2} \ge 1$ and
\begin{align} \label{eq: T r=2}
& i \frac{\pi}{2} {T}(k_{1}+1, k_{2}) - \frac{(-1)^{k_{2}/2} {D}_{k_{2}}}{k_{2}!} \left(\frac{\pi}{2}\right)^{k_{2}} \tilde{T}(k_{1}+1,1) \\ \nonumber
& + \sum_{\substack{j=0\\j:{\rm even}}}^{k_{2}-1}\frac{(-1)^{j/2}{D}_{j}}{j!}\left(\frac{\pi}{2}\right)^{j} {T}(k_{1}+1, k_{2}+1-j) \\ \nonumber
&+ \sum_{\substack{j=0\\j:{\rm even}}}^{k_{2}}\frac{(-1)^{j/2}{C}_{j}}{j!}\left(\frac{\pi}{2}\right)^{j} \tilde{T}(k_{1}+1, k_{2}+1-j)  \\ \nonumber
& = i^{-k_{1}-k_{2}} \mathcal{T}({D}_{n};k_{1}+1, k_{2}+1;0) + \sum_{h=0}^{k_{1}} i^{1-{k}_{2}-h} \tilde{T}({k}_{1} +1 - h) \mathcal{T}({D}_{n};{k}_{2}+1;h)
\end{align}
for $k_{2} \ge 2$, respectively. By taking the imaginary parts of \eqref{eq: Tt r=2} and \eqref{eq: T r=2}, we have 
\begin{align} \label{eq: TTt r=2 odd imag}
&\Im\left(i^{-k_{1}-k_{2}} \mathcal{T}(a_{n};k_{1}+1, k_{2}+1;0) + \sum_{h=0}^{k_{1}} i^{1-{k}_{2}-h} \tilde{T}({k}_{1} + 1 - h) \mathcal{T}(a_{n};{k}_{2}+1;h)\right)\\ \nonumber
&=\begin{cases} -(\pi/2) \tilde{T}(k_{1}+1,k_{2})  &(a_{n} = {C}_{n}),\\
+(\pi/2) T(k_{1}+1,k_{2}) &(a_{n} = {D}_{n}).
\end{cases}
\end{align}
If $k_{1} + k_{2}$ is even, the following theorem follows from Lemma \ref{lem; TtcalT} with $\mathbf{k}=(k_{1}+1,k_{2}+1)$.
\begin{theorem}\label{th:TCorD} When $k_{1} + k_{2}$ is even, we have
\begin{align*}
&\begin{cases} -\tilde{T}(k_{1}+1,k_{2})  &(a_{n} = {C}_{n}),\\
T(k_{1}+1,k_{2}) &(a_{n} = {D}_{n}).
\end{cases}\\
&= i^{ - k_{1} - {k}_{2}} \sum_{m=0}^{k_{1}} \sum_{n=0}^{k_{2}-1}  \frac{E_{m} a_{n}}{m!n!} \left(\frac{\pi}{2}\right)^{m+n}\tilde{T}(\{1\}^{k_{2}-n-1}, 2, \{1\}^{k_{1}-m})\\
&= i^{ - k_{1} - k_{2}} \sum_{\substack{l_{1}+{l}_{2} = k_{1}}}  \sum_{\substack{m_{1}+ {m}_{2} = k_{2} \\ m_{2}>0}} \binom{l_{1}+m_{1}}{l_{1}} \binom{l_{2}+m_{2}}{l_{2}} \\
&\quad \times \left(\sum_{h=0}^{l_{1}} \binom{l_{1}}{h} {E}_{k_{1}-h} \right) \left(\sum_{n=1}^{m_{2}} \binom{m_{2}}{n} a_{k_{2}-n} \right) \tilde{T}(\{1\}^{l_{2} + m_{2}-1}, 2, \{1\}^{l_{1} + m_{1}}).
\end{align*}
\end{theorem}
Here, we provide a proof of Theorem \ref{th:Tt+T=EG} in Section \ref{se:intro}.
\begin{proof}[Proof of Theorem \ref{th:Tt+T=EG}]
From Theorem \ref{th:TCorD}, it follows that
\begin{align*}
&{T}(k_{1}+1,k_{2}) + \tilde{T}(k_{1}+1,k_{2})\\
&=- i^{ - k_{1} - k_{2}} \sum_{\substack{l_{1}+{l}_{2} = k_{1}}}  \sum_{\substack{m_{1}+ {m}_{2} = k_{2} \\ m_{2}>0}} \binom{l_{1}+m_{1}}{l_{1}} \binom{l_{2}+m_{2}}{l_{2}} \\
&\quad \times \left(\sum_{h=0}^{l_{1}} \binom{l_{1}}{h} {E}_{k_{1}-h} \right) \left(\sum_{n=1}^{m_{2}} \binom{m_{2}}{n} G_{k_{2}-n} \right) \tilde{T}(\{1\}^{l_{2} + m_{2}-1}, 2, \{1\}^{l_{1} + m_{1}}).
\end{align*}
From Theorem \ref{th:BE=} and Theorem \ref{th:BG=}, we can see that
\begin{align*}
\sum_{h=0}^{l_{1}} \binom{l_{1}}{h} {E}_{k_{1}-h} 
&=  (-1)^{\lfloor (k_{1}-1)/2 \rfloor}\mathbb{E}(k_{1},l_{1}), \\
\sum_{n=1}^{m_{2}} \binom{m_{2}}{n} G_{k_{2}-n} 
&= (-1)^{\lfloor (k_{2}-1)/2 \rfloor}\mathbb{G}(k_{2},m_{2})
\end{align*}
for odd $k_{1}$ and odd $k_{2}$, respectively. Therefore, for odd $k_{1}$ and odd $k_{2}$, we have
\begin{align*}
&{T}(k_{1}+1,k_{2}) + \tilde{T}(k_{1}+1,k_{2})\\
&=\sum_{\substack{l_{1}+{l}_{2} = k_{1} \\ l_{1}>0}}  \sum_{\substack{m_{1}+ {m}_{2} = k_{2} \\ m_{2}>0}} \binom{l_{1}+m_{1}}{l_{1}} \binom{l_{2}+m_{2}}{l_{2}} \\
&\quad \times \mathbb{E}(k_{1},l_{1})\, \mathbb{G}(k_{2},m_{2})\, \tilde{T}(\{1\}^{l_{2} + m_{2}-1}, 2, \{1\}^{l_{1} + m_{1}})\\
&=\sum_{l_{1}=1}^{k_{1}}  \sum_{m_{2}=1}^{k_{2}} \binom{l_{1} + k_{2}-m_{2}}{l_{1}} \binom{k_{1}-l_{1}+m_{2}}{m_{2}} \\
&\quad \times \mathbb{E}(k_{1},l_{1})\, \mathbb{G}(k_{2},m_{2})\, \tilde{T}(\{1\}^{k_{1}-l_{1}+m_{2} -1}, 2, \{1\}^{ l_{1} + k_{2}-m_{2}}).
\end{align*}
This completes the proof of Theorem \ref{th:Tt+T=EG}.
\end{proof}
If $k_{1} + k_{2}$ is odd, we can obtain simple relations only for $k_{1}=0$. We will provide a proof of Theorem \ref{th:Tt+T=G}.
\begin{proof}[Proof of Theorem \ref{th:Tt+T=G}]
When $k_{1}$ is equal to $0$ and $k_{2}$ is a odd number, we can obtain 
\begin{align} \label{eq: Tortt1k_{2}}
&i^{-1-k_{2}} \mathcal{T}(a_{n};1, k_{2}+1;0)=\begin{cases} -(\pi/2) \tilde{T}(1,k_{2})  &(a_{n} = {C}_{n}),\\
+(\pi/2) T(1,k_{2}) &(a_{n} = {D}_{n})
\end{cases}
\end{align}
from equation \eqref{eq: TTt r=2 odd imag}. From Definition \ref{def;calT} and Lemma \ref{lem;calT=CT}, we obtain
\begin{align*}
&\begin{cases} -(\pi/2) \tilde{T}(1,k_{2})  & (a_{n} = {C}_{n}),\\
+(\pi/2) T(1,k_{2}) & (a_{n} = {D}_{n})
\end{cases}\\
&=i^{-1-k_{2}} \sum_{n=0}^{k_{2}-1} \frac{a_{n}}{n!} \left(\frac{\pi}{2}\right)^{n}\tilde{T}(\{1\}^{k_{2}-n-2},3)\\
&=i^{-1-k_{2}} \sum_{m_{2}=1}^{k_{2}}\left(\sum_{n=1}^{m_{2}} \binom{m_{2}}{n} a_{k_{2}-n} \right) \sum_{m_{1} = m_{2}}^{k_{2}}  \tilde{T}(\underbrace{1,\dots,1,1\hspace{-0mm}\overset{\tiny m_{2}}{\check{+1}}\hspace{-0mm},1,\dots,1,1\hspace{-0mm}\overset{\tiny m_{1}}{\check{+1}}\hspace{-0mm},1,\dots,1}_{k_{2}}).
\end{align*}
Theorem \ref{th:Tt+T=G} can be obtained by applying Theorem \ref{th:BG=} to the difference between the equations in each case.
\end{proof}
The left-hand side of the equations \eqref{eq: Tortt1k_{2}} has no $\pi/2$ factor. However, the difference between the equations in each case has a $\pi/2$ factor because
\begin{align*}
\mathcal{T}({G}_{n};\mathbf{k};h)
&=\sum_{n=0}^{k_{r}-2} \frac{G_{n}}{n!} \left(\frac{\pi}{2}\right)^{n}\tilde{T}((\mathbf{k}^{(n)})^{\dag},\{1\}^{h})\\
&=\frac{\pi}{2}\sum_{n=0}^{k_{r}-3} \frac{G_{n+1}}{(n+1)!} \left(\frac{\pi}{2}\right)^{n}\tilde{T}((\mathbf{k}^{(n+1)})^{\dag},\{1\}^{h}) \\
&= \frac{\pi}{2}\mathcal{T}({G}_{n+1}/(n+1);\mathbf{k}^{(1)};h)
\end{align*}
holds for any $\mathbf{k} = (k_{1},\dots,k_{r})$ with $k_{r} \ge 3$. By this observation, we have \begin{align*}
&i^{-1-k_{2}} \mathcal{T}(G_{n+1}/(n+1);1, k_{2};0) = - T(1,k_{2}) - \tilde{T}(1,k_{2}).
\end{align*}
Therefore, the following theorem follows.
\begin{theorem} For odd $k_{2} \ge 3$, we have
\begin{align*}
&T(1,k_{2}) + \tilde{T}(1,k_{2})\\ \nonumber
&=i^{1-k_{2}} \sum_{n=0}^{k_{2}-2} \frac{G_{n+1}}{(n+1)!} \left(\frac{\pi}{2}\right)^{n}\tilde{T}(\{1\}^{k_{2}-n-2},3)\\ \nonumber
&=i^{1-k_{2}} \sum_{m_{2}=1}^{k_{2}-1} \left(\sum_{n=1}^{m_{2}} \binom{m_{2}}{n} \frac{G_{k_{2}-n}}{k_{2}-n} \right) \sum_{m_{1}=m_{2}}^{k_{2}-1} \tilde{T}(\underbrace{1,\dots,1,1\hspace{-0mm}\overset{\tiny m_{2}}{\check{+1}}\hspace{-0mm},1,\dots,1,1\hspace{-0mm}\overset{\tiny m_{1}}{\check{+1}}\hspace{-0mm},1,\dots,1}_{k_{2}-1}).
\end{align*}
\end{theorem}
\subsection{Depth three}
In this subsection, we consider the case $r = 3$. In this case, we can obtain simple relations only for special indices. The first is the case $\mathbf{k} = (k_{1}+1,1,k_{3}+1)$. 
\begin{theorem}Let $k_{1}$ be a non-negative integer and $k_{3}$ a positive integer. If $k_{1}+k_{3}$ is even, we have
\begin{align*}
& \tilde{T}(k_{1}+1,1,k_{3}+1) - \delta_{k_{3}:even}\frac{\mathbb{E}_{k_{3}}}{k_{3}!} \left(\frac{\pi}{2}\right)^{k_{3}} \tilde{T}(k_{1}+1,1,1)\\
&- \delta_{k_{3}: \textup{odd}} \tilde{T}({k}_{1}+1,1) \tilde{T}({k}_{3}+1)\\
&=i^{-{k}_{1}-k_{3}} \sum_{m=0}^{k_{1}} \sum_{n=0}^{k_{3}-1} \frac{E_{m}}{m!} \frac{E_{n}}{n!}\left(\frac{\pi}{2}\right)^{m+n+1}\tilde{T}(\{1\}^{k_{3}-1-n},3,\{1\}^{k_{1}-m})\\
&=i^{-{k}_{1}-k_{3}} \frac{\pi}{2} \sum_{\substack{l_{1}+l_{2}+l_{3} = k_{1}}}  \sum_{\substack{m_{1}+m_{2}+m_{3} = k_{3} \\ m_{3} > 0}} \left( \sum_{h=0}^{l_{1}} \binom{l_{1}}{h} {E}_{k_{1}-h} \right) \left(\sum_{n=1}^{m_{3}} \binom{m_{3}}{n} E_{k_{3}-n} \right) \\
&\quad \times \binom{l_{1}+m_{1}}{l_{1}} \binom{l_{2}+m_{2}}{l_{2}} \binom{l_{3}+m_{3}}{l_{3}}  \tilde{T}(\underbrace{1,\dots,1,1\hspace{-2mm}\overset{\tiny l_{3}+m_{3}}{\check{+1}}\hspace{-2mm},1,\dots,1,1\hspace{-7mm}\overset{\tiny l_{3}+m_{3}+l_{2}+m_{2}}{\check{+1}}\hspace{-7mm},1,\dots,1}_{k_{1}+k_{3}}).
\end{align*}
If $k_{1}+k_{3}$ is odd, we have
\begin{align*}
&T(k_{1}+1,1,k_{3}+1) + \delta_{k_{3}: \textup{odd}} \frac{\mathbb{E}_{k_{3}}}{k_{3}!} \left(\frac{\pi}{2}\right)^{k_{3}} \tilde{T}(k_{1}+1,1,1)\\
&+\delta_{k_{3}:even}\tilde{T}({k}_{1}+1,1) {T}({k}_{3}+1)\\
&=i^{1-{k}_{1}-k_{3}} \sum_{m=0}^{k_{1}} \sum_{n=0}^{k_{3}-1} \frac{E_{m}}{m!} \frac{F_{n}}{n!}\left(\frac{\pi}{2}\right)^{m+n+1}\tilde{T}(\{1\}^{k_{3}-1-n},3,\{1\}^{k_{1}-m})\\
&=i^{1-{k}_{1}-k_{3}} \frac{\pi}{2} \sum_{\substack{l_{1}+l_{2}+l_{3} = k_{1}}}  \sum_{\substack{m_{1}+m_{2}+m_{3} = k_{3} \\ m_{3} > 0}} \left( \sum_{h=0}^{l_{1}} \binom{l_{1}}{h} {E}_{k_{1}-h} \right) \left(\sum_{n=1}^{m_{3}} \binom{m_{3}}{n} F_{k_{3}-n} \right) \\
&\quad \times \binom{l_{1}+m_{1}}{l_{1}} \binom{l_{2}+m_{2}}{l_{2}} \binom{l_{3}+m_{3}}{l_{3}}  \tilde{T}(\underbrace{1,\dots,1,1\hspace{-2mm}\overset{\tiny l_{3}+m_{3}}{\check{+1}}\hspace{-2mm},1,\dots,1,1\hspace{-7mm}\overset{\tiny l_{3}+m_{3}+l_{2}+m_{2}}{\check{+1}}\hspace{-7mm},1,\dots,1}_{k_{1}+k_{3}}).
\end{align*}
\end{theorem}
\begin{proof}
We assume $\mathbf{k} = (k_{1}+1,1,k_{3}+1)$ with $k_{1} \ge 0$ and $k_{3} \ge 1$. In this case, Theorem \ref{th: =E} and Theorem \ref{th: =F} reduce to
\begin{align} \label{eq: Tt r=3}
& i \tilde{T}(k_{1}+1,1,k_{3}+1) - i\delta_{k_{3}:even}\frac{\mathbb{E}_{k_{3}}}{k_{3}!} \left(\frac{\pi}{2}\right)^{k_{3}} \tilde{T}(k_{1}+1,1,1)\\ \nonumber
&+\sum_{\substack{j=0\\j:{\rm even}}}^{k_{3}-1}\frac{\mathbb{E}_{j}}{j!}\left(\frac{\pi}{2}\right)^{j}T(\mathbf{k}^{(j)}) + \sum_{\substack{j=0\\j:{\rm odd}}}^{k_{3}}\frac{\mathbb{E}_{j}}{j!} \left(\frac{\pi}{2}\right)^{j}\tilde{T}(\mathbf{k}^{(j)}) \\ \nonumber
&= i^{-k_{1}-k_{3}} \mathcal{T}(E_{n};\mathbf{k};0) + \sum_{h=0}^{k_{1}} i^{r-|\mathbf{k}^{1}|-h} \tilde{T}((\mathbf{k}_{1})^{(h)}) \mathcal{T}(E_{n};\mathbf{k}^{1};h)\\ \nonumber
&\quad + i^{2-{k}_{3}} \tilde{T}(k_{1}+1,1) \mathcal{T}(E_{n};k_{3}+1;0)
\end{align} 
and 
\begin{align} \label{eq: T r=3}
&T(k_{1}+1,1,k_{3}+1) + \delta_{k_{3}: \textup{odd}} \frac{\mathbb{E}_{k_{3}}}{k_{3}!} \left(\frac{\pi}{2}\right)^{k_{3}} \tilde{T}(k_{1}+1,1,1)\\ \nonumber
&+i \sum_{\substack{j=0\\j:{\rm odd}}}^{k_{3}-1}\frac{\mathbb{E}_{j}}{j!} \left(\frac{\pi}{2}\right)^{j} {T}(\mathbf{k}^{(j)}) + i\sum_{\substack{j=0\\j:{\rm even}}}^{k_{3}} \frac{\mathbb{E}_{j}}{j!} \left(\frac{\pi}{2}\right)^{j} \tilde{T}(\mathbf{k}^{(j)})\\ \nonumber
&= i^{-k_{1}-k_{3}} \mathcal{T}(F_{n};\mathbf{k};0) + \sum_{h=0}^{k_{1}} i^{r-|\mathbf{k}^{1}|-h} \tilde{T}((\mathbf{k}_{1})^{(h)}) \mathcal{T}(F_{n};\mathbf{k}^{1};h) \\ \nonumber
&\quad + i^{2-{k}_{3}} \tilde{T}(k_{1}+1,1) \mathcal{T}(F_{n};k_{3}+1;0),
\end{align}
respectively. 
If $k_{1}+k_{3}$ is even, by taking the imaginary part of \eqref{eq: Tt r=3}, we have
\begin{align*} 
& \tilde{T}(k_{1}+1,1,k_{3}+1) - \delta_{k_{3}: \textup{even}} \frac{\mathbb{E}_{k_{3}}}{k_{3}!} \left(\frac{\pi}{2}\right)^{k_{3}} \tilde{T}(k_{1}+1,1,1)\\
&- \delta_{k_{3}: \textup{odd}} i^{1-{k}_{3}} \tilde{T}({k}_{1}+1,1) \mathcal{T}(E_{n};{k}_{3}+1;0)\\
&=\Im\left(\sum_{h=0}^{k_{1}} i^{r-|\mathbf{k}^{1}|-h} \tilde{T}((\mathbf{k}_{1})^{(h)}) \mathcal{T}(E_{n};\mathbf{k}^{1};h) \right).
\end{align*}
The first part of the theorem can be obtained by combining this with equation \eqref{eq:Ttk_1} and Lemma \ref{lem; TtcalT}. If $k_{1}+k_{3}$ is odd, by taking the imaginary part of \eqref{eq: T r=3}, we have
\begin{align*}
&T(k_{1}+1,1,k_{3}+1) + \delta_{k_{3}: \textup{odd}} \frac{\mathbb{E}_{k_{3}}}{k_{3}!} \left(\frac{\pi}{2}\right)^{k_{3}} \tilde{T}(k_{1}+1,1,1)\\
&+\delta_{k_{3}: \textup{even}}i^{-{k}_{3}} \tilde{T}({k}_{1}+1,1) \mathcal{T}(F_{n};{k}_{3}+1;0)\\
&=\Re\left(\sum_{h=0}^{k_{1}} i^{r-|\mathbf{k}^{1}|-h} \tilde{T}((\mathbf{k}_{1})^{(h)}) \mathcal{T}(F_{n};\mathbf{k}^{1};h)\right).
\end{align*}
The latter part of the theorem can be obtained by combining this with equation \eqref{eq:Tk_1} and Lemma \ref{lem; TtcalT}.
\end{proof}
Note that Theorem \ref{th:TorTt1=EE} in Section \ref{se:intro} readily follows from this theorem and Theorem \ref{th:BE=}. Next, we prove Theorem \ref{th:TorTt=E}.
\begin{proof}[Proof of Theorem \ref{th:TorTt=E}]
If $k_{1}$ is equal to $0$ and $k_{3}$ is a positive odd integer, by taking the imaginary part of  \eqref{eq: Tt r=3}, we obtain
\begin{align*} 
& \tilde{T}(1,1,k_{3}+1) \\
&= i^{-1-k_{3}} \mathcal{T}(E_{n};1,1,k_{3}+1;0) + i^{1-{k}_{3}} \tilde{T}(1,1) \mathcal{T}(E_{n};k_{3}+1;0).
\end{align*}
By applying Definition \ref{def;calT}, Lemma \ref{lem;calT=CT} and equation \eqref{eq:Ttk_1} to this equation, we obtain
\begin{align*} 
& \tilde{T}(1,1,k_{3}+1) - \tilde{T}(1,1) \tilde{T}(k_{3}+1) \\
&= i^{-1-k_{3}} \sum_{n=0}^{k_{3}-1} \frac{E_{n}}{n!} \left(\frac{\pi}{2}\right)^{n}\tilde{T}(\{1\}^{k_{3}-1-n},4)\\
&= i^{-1-k_{3}} \sum_{m_{3}=1}^{k_{3}} \left(\sum_{n=1}^{m_{3}} \binom{m_{3}}{n} E_{k_{3}-n} \right) \\
&\quad \times \sum_{m_{2} = m_{3}}^{k_{2}}  \sum_{m_{1} = m_{2}}^{k_{3}}\tilde{T}(\underbrace{1,\dots,1,1\hspace{-0mm}\overset{\tiny m_{3}}{\check{+1}}\hspace{-0mm},1,\dots,1,1\hspace{-0mm}\overset{\tiny m_{2}}{\check{+1}}\hspace{-0mm},1,\dots,1,1\hspace{-0mm}\overset{\tiny m_{1}}{\check{+1}}\hspace{-0mm},1,\dots,1}_{k_{3}}).
\end{align*}
On the other hand, if $k_{1}$ is equal to $0$ and $k_{3}$ is a positive even integer, by taking the real part of \eqref{eq: T r=3}, we obtain
\begin{align*} 
&T(1,1,k_{3}+1) \\
&= i^{-k_{3}} \mathcal{T}(F_{n};1,1,k_{3}+1;0) - i^{-{k}_{3}} \tilde{T}(1,1) \mathcal{T}(F_{n};k_{3}+1;0).
\end{align*}
By applying Definition \ref{def;calT}, Lemma \ref{lem;calT=CT} and the equatioin \eqref{eq:Tk_1} to this equation, we obtain
\begin{align*} 
&T(1,1,k_{3}+1) + \tilde{T}(1,1) {T}(k_{3}+1)\\
&= i^{-k_{3}} \sum_{n=0}^{k_{3}-1} \frac{F_{n}}{n!} \left(\frac{\pi}{2}\right)^{n}\tilde{T}(\{1\}^{k_{3}-1-n},4)\\
&= i^{-k_{3}} \sum_{m_{3}=1}^{k_{3}} \left(\sum_{n=1}^{m_{3}} \binom{m_{3}}{n} F_{k_{3}-n} \right) \\
&\quad \times \sum_{m_{2} = m_{3}}^{k_{2}}  \sum_{m_{1} = m_{2}}^{k_{3}}\tilde{T}(\underbrace{1,\dots,1,1\hspace{-0mm}\overset{\tiny m_{3}}{\check{+1}}\hspace{-0mm},1,\dots,1,1\hspace{-0mm}\overset{\tiny m_{2}}{\check{+1}}\hspace{-0mm},1,\dots,1,1\hspace{-0mm}\overset{\tiny m_{1}}{\check{+1}}\hspace{-0mm},1,\dots,1}_{k_{3}}).
\end{align*}
Finally, by using Theorem \ref{th:BE=}, Theorem \ref{th:TorTt=E} can be obtained.
\end{proof}
Next, we will consider the case $\mathbf{k} = (k_{1}+1,2,k_{3}+1)$ with $k_{1} \ge 0$ and $k_{3} \ge 1$ and proceed to prove Theorem \ref{th:TorTt2=EE}.
\begin{proof}[Proof of Theorem \ref{th:TorTt2=EE}]

We assume $\mathbf{k} = (k_{1}+1,2,k_{3}+1)$ with $k_{1} \ge 0$ and $k_{3} \ge 1$. In this case, Theorem \ref{th: =E} and Theorem \ref{th: =F} reduce to

\begin{align}\label{eq: Tt r=3 2}
&i \tilde{T}(k_{1}+1,2,k_{3}+1) - i\delta_{k_{3}: \textup{even}}\frac{\mathbb{E}_{k_{3}}}{k_{3}!} \left(\frac{\pi}{2}\right)^{k_{3}} \tilde{T}(k_{1}+1,2,1)\\ \nonumber
&+\sum_{\substack{j=0\\j:{\rm even}}}^{k_{3}-1}\frac{\mathbb{E}_{j}}{j!}\left(\frac{\pi}{2}\right)^{j}T(\mathbf{k}^{(j)}) + \sum_{\substack{j=0\\j:{\rm odd}}}^{k_{3}}\frac{\mathbb{E}_{j}}{j!} \left(\frac{\pi}{2}\right)^{j}\tilde{T}(\mathbf{k}^{(j)}) \\ \nonumber
&= i^{-k_{1}-k_{3}-1} \mathcal{T}(E_{n};\mathbf{k};0) + \sum_{h=0}^{k_{1}} i^{r-|\mathbf{k}^{j}|-h} \tilde{T}((\mathbf{k}_{j})^{(h)}) \mathcal{T}(E_{n};\mathbf{k}^{j};h)\\ \nonumber
&\quad + i^{2-k_{3}} \tilde{T}(k_{1}+1,2) \mathcal{T}(E_{n};k_{3}+1;0)+ i^{1-k_{3}} \tilde{T}(k_{1}+1,1) \mathcal{T}(E_{n};k_{3}+1;1)
\end{align}
and
\begin{align}\label{eq: T r=3 2}
&T(k_{1}+1,2,k_{3}+1) + \delta_{k_{3}: \textup{odd}} \frac{\mathbb{E}_{k_{3}}}{k_{3}!} \left(\frac{\pi}{2}\right)^{k_{3}} \tilde{T}(k_{1}+1,2,1)\\ \nonumber
&+i \sum_{\substack{j=0\\j:{\rm odd}}}^{k_{3}-1}\frac{\mathbb{E}_{j}}{j!} \left(\frac{\pi}{2}\right)^{j} {T}(\mathbf{k}^{(j)}) + i\sum_{\substack{j=0\\j:{\rm even}}}^{k_{3}} \frac{\mathbb{E}_{j}}{j!} \left(\frac{\pi}{2}\right)^{j} \tilde{T}(\mathbf{k}^{(j)})\\ \nonumber
&= i^{-k_{1}-k_{3}-1} \mathcal{T}(F_{n};\mathbf{k};0) + \sum_{h=0}^{k_{1}} i^{r-|\mathbf{k}^{j}|-h} \tilde{T}((\mathbf{k}_{j})^{(h)}) \mathcal{T}(F_{n};\mathbf{k}^{j};h)\\ \nonumber
&\quad + i^{2-k_{3}} \tilde{T}(k_{1}+1,2) \mathcal{T}(F_{n};k_{3}+1;0)+ i^{1-k_{3}} \tilde{T}(k_{1}+1,1) \mathcal{T}(F_{n};k_{3}+1;1),
\end{align}
respectively. 

We assume $k_{1} \ge 0$ is an even number. Taking $k_{3}$ as an odd number and the imaginary part of equation \eqref{eq: Tt r=3 2}, we obtain
\begin{align*}
& \tilde{T}(k_{1}+1,2,k_{3}+1) \\
&= \Im\left(\sum_{h=0}^{k_{1}} i^{r-|\mathbf{k}^{j}|-h} \tilde{T}((\mathbf{k}_{j})^{(h)}) \mathcal{T}(E_{n};\mathbf{k}^{j};h) \right) + i^{1-k_{3}} \tilde{T}(k_{1}+1,2) \mathcal{T}(E_{n};k_{3}+1;0).
\end{align*}
By applying equation \eqref{eq:Ttk_1} and Lemma \ref{lem; TtcalT} to this equation, we obtain
\begin{align*}
& \tilde{T}(k_{1}+1,2,k_{3}+1) - \tilde{T}(k_{1}+1,2) \tilde{T}(k_{3}+1)\\
&=i^{-k_{1}-k_{3}-1} \sum_{m=0}^{k_{1}} \sum_{n=0}^{k_{3}-1} \frac{E_{m}}{m!} \frac{E_{n}}{n!}\left(\frac{\pi}{2}\right)^{m+n+1}\tilde{T}(\{1\}^{k_{3}-n-1},2,2,\{1\}^{k_{1}-m})\\
&=i^{-k_{1}-k_{3}-1} \frac{\pi}{2} \sum_{\substack{l_{1}+l_{2}+l_{3}= k_{1}}}  \sum_{\substack{m_{1}+m_{2}+m_{3} = k_{3} \\ m_{3} > 0}} \left( \sum_{h=0}^{l_{1}} \binom{l_{1}}{h} {E}_{k_{1}-h} \right) \left(\sum_{n=1}^{m_{3}} \binom{m_{3}}{n} E_{k_{3}-n} \right) \\
&\quad \times \binom{l_{1}+m_{1}}{l_{1}} \binom{1+l_{2}+m_{2}}{1,l_{2},m_{2}} \binom{l_{3}+m_{3}}{l_{3}} \tilde{T}(\underbrace{1,\dots,1,\hspace{-3mm}\overset{\tiny l_{3}+m_{3}}{\check{2}}\hspace{-3mm},1,\dots,1,\hspace{-10mm}\overset{\tiny l_{3}+m_{3}+l_{2}+m_{2}+1}{\check{2}}\hspace{-10mm},1,\dots,1}_{k_{1}+k_{3}+1}).
\end{align*}
Keeping $k_{1}$ even, and taking $k_{3}$ as an even number and the real part of equation \eqref{eq: T r=3 2}, we obtain
\begin{align*}
&T(k_{1}+1,2,k_{3}+1) \\
&= \Re\left(\sum_{h=0}^{k_{1}} i^{r-|\mathbf{k}^{j}|-h} \tilde{T}((\mathbf{k}_{j})^{(h)}) \mathcal{T}(F_{n};\mathbf{k}^{j};h)\right) - i^{-k_{3}} \tilde{T}(k_{1}+1,2) \mathcal{T}(F_{n};k_{3}+1;0).
\end{align*}
By applying equation \eqref{eq:Tk_1} and Lemma \ref{lem; TtcalT} to this equation, we obtain
\begin{align*}
&T(k_{1}+1,2,k_{3}+1) + \tilde{T}(k_{1}+1,2) {T}(k_{3}+1)\\
&=i^{-k_{1}-k_{3}} \sum_{m=0}^{k_{1}} \sum_{n=0}^{k_{3}-1} \frac{E_{m}}{m!} \frac{F_{n}}{n!}\left(\frac{\pi}{2}\right)^{m+n+1}\tilde{T}(\{1\}^{k_{3}-n-1},2,2,\{1\}^{k_{1}-m})\\
&=i^{-k_{1}-k_{3}} \frac{\pi}{2} \sum_{\substack{l_{1}+l_{2}+l_{3}= k_{1}}}  \sum_{\substack{m_{1}+m_{2}+m_{3} = k_{3} \\ m_{3} > 0}} \left( \sum_{h=0}^{l_{1}} \binom{l_{1}}{h} {E}_{k_{1}-h} \right) \left(\sum_{n=1}^{m_{3}} \binom{m_{3}}{n} F_{k_{3}-n} \right) \\
&\quad \times \binom{l_{1}+m_{1}}{l_{1}} \binom{1+l_{2}+m_{2}}{1,l_{2},m_{2}} \binom{l_{3}+m_{3}}{l_{3}} \tilde{T}(\underbrace{1,\dots,1,\hspace{-3mm}\overset{\tiny l_{3}+m_{3}}{\check{2}}\hspace{-3mm},1,\dots,1,\hspace{-10mm}\overset{\tiny l_{3}+m_{3}+l_{2}+m_{2}+1}{\check{2}}\hspace{-10mm},1,\dots,1}_{k_{1}+k_{3}+1}).
\end{align*}
Therefore, combining these with Theorem \ref{th:BE=} and Theorem \ref{th:Ekk-j}  proves Theorem \ref{th:TorTt2=EE}.
\end{proof}

\section*{Acknowledgment}
The author is deeply grateful to Prof. Masanobu Kaneko and Prof. Kohji Matsumoto for their helpful comments.
\appendix
\section{Tables} \label{se:Appendix}
The following are the first few values of the numbers that are dealt with in this paper.
{\footnotesize
\renewcommand*{\arraystretch}{1}
\begin{table}[H]
\caption{$\mathbb{E}(k,j)$}
\centering
  \begin{tabular}{|c|c|c|c|c|c|c|c|c|c|c|c|c|c|} \hline
\backslashbox{$k$}{$j$}&$0$&$1$&$2$&$3$&$4$&$5$&$6$&$7$&$8$&$9$&$10$\\ \hline 
$0$&$1$& & & & & & & & & &\\ \hline 
$1$&$0$&$1$ & & & & & & & & &\\ \hline 
$2$&$0$&$1$&$1$&  &  &  &  &  &  &  &  \\ \hline
$3$&$0$&$1$&$2$&$2$&  &  &  &  &  &  &  \\ \hline
$4$&$0$&$2$&$4$&$5$&$5$&  &  &  &  &  &  \\ \hline
$5$&$0$&$5$&$10$&$14$&$16$&$16$&  &  &  &  &  \\ \hline
$6$&$0$&$16$&$32$&$46$&$56$&$61$&$61$&  &  &  &  \\ \hline
$7$&$0$&$61$&$122$&$178$&$224$&$256$&$272$&$272$&  &  &  \\ \hline
$8$&$0$&$272$&$544$&$800$&$1024$&$1202$&$1324$&$1385$&$1385$&  &  \\ \hline
$9$&$0$&$1385$&$2770$&$4094$&$5296$&$6320$&$7120$&$7664$&$7936$&$7936$& \\ \hline
$10$&$0$&$7936$&$15872$&$23536$&$30656$&$36976$&$42272$&$46366$&$49136$&$50521$&$50521$  \\ \hline
  \end{tabular}
\end{table}
\renewcommand*{\arraystretch}{1}
\begin{table}[H]
\caption{$\mathbb{G}(k,j)$}
\centering
  \begin{tabular}{|c|c|c|c|c|c|c|c|c|c|c|c|c|c|} \hline
\backslashbox{$k$}{$j$}&$0$&$1$&$2$&$3$&$4$&$5$&$6$&$7$&$8$&$9$&$10$\\ \hline 
$0$&$0$& & & & & & & & & &\\ \hline 
$1$&$0$&$0$ & & & & & & & & &\\ \hline 
$2$&$0$&$1$&$0$&  &  &  &  &  &  &  &  \\ \hline
$3$&$0$&$1$&$1$&$0$&  &  &  &  &  &  &  \\ \hline
$4$&$0$&$0$&$1$&$0$&$0$&  &  &  &  &  &  \\ \hline
$5$&$0$&$1$&$2$&$2$&$1$&$0$&  &  &  &  &  \\ \hline
$6$&$0$&$0$&$1$&$1$&$1$&$0$&$0$&  &  &  &  \\ \hline
$7$&$0$&$3$&$6$&$8$&$8$&$6$&$3$&$0$&  &  &  \\ \hline
$8$&$0$&$0$&$3$&$3$&$5$&$3$&$3$&$0$&$0$&  &  \\ \hline
$9$&$0$&$17$&$34$&$48$&$56$&$56$&$48$&$34$&$17$&$0$& \\ \hline
$10$&$0$&$0$&$17$&$17$&$31$&$25$&$31$&$17$&$17$&$0$&$0$  \\ \hline
  \end{tabular}
\end{table}
\renewcommand*{\arraystretch}{1}
\begin{table}[H]
\caption{$E(k,j)$}
\centering
  \begin{tabular}{|c|c|c|c|c|c|c|c|c|c|c|c|c|c|} \hline
\backslashbox{$k$}{$j$}&$0$&$1$&$2$&$3$&$4$&$5$&$6$&$7$&$8$&$9$&$10$\\ \hline 
$0$&$0$& & & & & & & & & &\\ \hline 
$1$&$0$&$1$ & & & & & & & & &\\ \hline 
$2$&$0$&$0$&$1$&  &  &  &  &  &  &  &  \\ \hline
$3$&$0$&$1$&$2$&$2$&  &  &  &  &  &  &  \\ \hline
$4$&$0$&$0$&$1$&$3$&$5$&  &  &  &  &  &  \\ \hline
$5$&$0$&$5$&$10$&$14$&$16$&$16$&  &  &  &  &  \\ \hline
$6$&$0$&$0$&$5$&$15$&$29$&$45$&$61$&  &  &  &  \\ \hline
$7$&$0$&$61$&$122$&$178$&$224$&$256$&$272$&$272$&  &  &  \\ \hline
$8$&$0$&$0$&$61$&$183$&$361$&$585$&$841$&$1113$&$1385$&  &  \\ \hline
$9$&$0$&$1385$&$2770$&$4094$&$5296$&$6320$&$7120$&$7664$&$7936$&$7936$& \\ \hline
$10$&$0$&$0$&$1385$&$4155$&$8249$&$13545$&$19865$&$26985$
&$34649$&$42585$&$50521$  \\ \hline
  \end{tabular}
\end{table}
\renewcommand*{\arraystretch}{1}
\begin{table}[H]
\caption{$F(k,j)$}
\centering
  \begin{tabular}{|c|c|c|c|c|c|c|c|c|c|c|c|c|c|} \hline
\backslashbox{$k$}{$j$}&$0$&$1$&$2$&$3$&$4$&$5$&$6$&$7$&$8$&$9$&$10$\\ \hline 
$0$&$0$& & & & & & & & & &\\ \hline 
$1$&$0$&$1$ & & & & & & & & &\\ \hline 
$2$&$0$&$1$&$1$&  &  &  &  &  &  &  &  \\ \hline
$3$&$0$&$0$&$1$&$2$&  &  &  &  &  &  &  \\ \hline
$4$&$0$&$2$&$4$&$5$&$5$&  &  &  &  &  &  \\ \hline
$5$&$0$&$0$&$2$&$6$&$11$&$16$&  &  &  &  &  \\ \hline
$6$&$0$&$16$&$32$&$46$&$56$&$61$&$61$&  &  &  &  \\ \hline
$7$&$0$&$0$&$16$&$48$&$94$&$150$&$211$&$272$&  &  &  \\ \hline
$8$&$0$&$272$&$544$&$800$&$1024$&$1202$&$1324$&$1385$&$1385$&  &  \\ \hline
$9$&$0$&$0$&$272$&$816$&$1616$&$2640$&$3842$&$5166$&$6551$&$7936$& \\ \hline
$10$&$0$&$7936$&$15872$&$23536$&$30656$&$36976$&$42272$&$46366$
&$49136$&$50521$&$50521$  \\ \hline
  \end{tabular}
\end{table}
\renewcommand*{\arraystretch}{1}
\begin{table}[H]
\caption{$G(k,j)$}
\centering
  \begin{tabular}{|c|c|c|c|c|c|c|c|c|c|c|c|c|c|} \hline
\backslashbox{$k$}{$j$}&$0$&$1$&$2$&$3$&$4$&$5$&$6$&$7$&$8$&$9$&$10$\\ \hline 
$0$&$0$& & & & & & & & & &\\ \hline 
$1$&$0$&$0$ & & & & & & & & &\\ \hline 
$2$&$0$&$1$&$2$&  &  &  &  &  &  &  &  \\ \hline
$3$&$0$&$1$&$1$&$0$&  &  &  &  &  &  &  \\ \hline
$4$&$0$&$0$&$1$&$2$&$2$&  &  &  &  &  &  \\ \hline
$5$&$0$&$1$&$2$&$2$&$1$&$0$&  &  &  &  &  \\ \hline
$6$&$0$&$0$&$1$&$3$&$5$&$6$&$6$&  &  &  &  \\ \hline
$7$&$0$&$3$&$6$&$8$&$8$&$6$&$3$&$0$&  &  &  \\ \hline
$8$&$0$&$0$&$3$&$9$&$17$&$25$&$31$&$34$&$34$&  &  \\ \hline
$9$&$0$&$17$&$34$&$48$&$56$&$56$&$48$&$34$&$17$&$0$& \\ \hline
$10$&$0$&$0$&$17$&$51$&$99$&$155$&$211$&$259$&$293$&$310$&$310$  \\ \hline
  \end{tabular}
\end{table}
\renewcommand*{\arraystretch}{1}
\begin{table}[H]
\caption{$H(k,j)$}
\centering
  \begin{tabular}{|c|c|c|c|c|c|c|c|c|c|c|c|c|c|} \hline
\backslashbox{$k$}{$j$}&$0$&$1$&$2$&$3$&$4$&$5$&$6$&$7$&$8$&$9$&$10$\\ \hline 
$0$&$0$& & & & & & & & & &\\ \hline 
$1$&$0$&$0$ & & & & & & & & &\\ \hline 
$2$&$0$&$1$&$0$&  &  &  &  &  &  &  &  \\ \hline
$3$&$0$&$1$&$-1$&$0$&  &  &  &  &  &  &  \\ \hline
$4$&$0$&$0$&$1$&$0$&$0$&  &  &  &  &  &  \\ \hline
$5$&$0$&$1$&$0$&$0$&$-1$&$0$&  &  &  &  &  \\ \hline
$6$&$0$&$0$&$1$&$1$&$1$&$0$&$0$&  &  &  &  \\ \hline
$7$&$0$&$3$&$0$&$2$&$-2$&$0$&$-3$&$0$&  &  &  \\ \hline
$8$&$0$&$0$&$3$&$3$&$5$&$3$&$3$&$0$&$0$&  &  \\ \hline
$9$&$0$&$17$&$0$&$14$&$-6$&$6$&$-14$&$0$&$-17$&$0$& \\ \hline
$10$&$0$&$0$&$17$&$17$&$31$&$25$&$31$&$17$&$17$&$0$&$0$  \\ \hline
  \end{tabular}
\end{table}
}

\end{document}